\documentclass[a4paper]{article}

\usepackage[pages=all, color=black, position={current page.south}, placement=bottom, scale=1, opacity=1, vshift=5mm]{background}
 \SetBgContents{
 }      

\usepackage[margin=1in]{geometry} 

\usepackage{amsmath}
\usepackage{amsthm}
\usepackage{amssymb}

\usepackage[utf8]{inputenc}
\usepackage{hyperref}
\hypersetup{
	unicode,
	pdfauthor={Author One, Author Two, Author Three},
	pdftitle={A simple article template},
	pdfsubject={A simple article template},
	pdfkeywords={article, template, simple},
	pdfproducer={LaTeX},
	pdfcreator={pdflatex}
}

\usepackage{amsthm}
\usepackage{latexsym}
\usepackage{geometry}
\usepackage{amsfonts}
\usepackage{tikz-cd}
\usepackage{amssymb}
\usepackage{mathrsfs}
\usepackage{mathtools}
\usepackage[english]{babel}
\usepackage{calligra}
\usepackage[T1]{fontenc}
\usetikzlibrary{arrows}
\usepackage{latexsym}
\usepackage{wasysym}
\usepackage{xcolor}
\usepackage{tikz-cd}
\usepackage{amsfonts}
\usepackage{amssymb}
\usepackage{amsthm}
\usepackage{amsmath}
\usepackage{mathrsfs}
\usepackage{indentfirst}
\usepackage{eqnarray}
\usepackage{graphicx}
\graphicspath{ {./images/} }
\usepackage{caption}
\usepackage{wrapfig}
\usepackage{float}

\usepackage{amsmath,amssymb}
\usepackage{tikz}
\usetikzlibrary{decorations.pathreplacing}

\usepackage{ytableau}
\usepackage{amssymb,latexsym,comment,url}
\usepackage{amsfonts}
\usepackage{amsthm}
\usepackage{enumerate}
\usepackage[mathscr]{eucal}
\usepackage{xcolor}
\usepackage{eqlist}

\usepackage[sort&compress,numbers,square]{natbib}
\theoremstyle{plain}
\newtheorem{theorem}{Theorem}

\newtheorem{lemma}[theorem]{Lemma}

\theoremstyle{definition}
\newtheorem{definition}[theorem]{Definition}

\usepackage{graphicx, color}
\graphicspath{{fig/}}

\usepackage{algorithm, algpseudocode} 
\usepackage{mathrsfs} 

\title{Method of Weighted Words on Cylindric Partitions}
\author{Burcu Barsakçı}

\date{Sabancı University, İstanbul, Turkey \\  
    \texttt{burcubarsakci@sabanciuniv.edu}\\[2ex]%
}

\begin{document}
	\maketitle
	
	\begin{abstract}
\noindent We study the generating functions of cylindric partitions having profile $c=(c_1, c_2, \ldots, c_r)$ with rank $2$ and levels $2, 3$ and $4$. As a result, we give expressions alternative to Borodin's formula for these generating functions. We use the method of weighted words which was first introduced by Alladi and Gordon, later was applied by Dousse in a new version to prove some partition identities and to get infinite products. We adapt the method to our subject with a more combinatorial approach. \\[1pt]

        \\[1pt]
        
        
	\end{abstract}

	
	\section{Introduction}
	\label{sec:intro}
Cylindric partitions are given by the following definition due to Gessel and Krattenthaler \cite{gessel1997cylindric}. 
\begin{definition} \label{def cyc}
 Let $r$ and $\ell$ be positive integers. Let $c=(c_1, c_2, \ldots, c_r)$ be a composition where $c_1 + c_2 + \ldots + c_r = \ell$. A cylindric partition with profile $c$ is a vector partition $\Lambda=(\lambda^{(1)}, \lambda^{(2)}, \ldots, \lambda^{(r)})$ where each $\lambda^{(i)}$ is an ordinary partition such that $\lambda^{(i)}=(\lambda_1^{(i)} , \lambda_2^{(i)}, \ldots , \lambda_{s_i}^{(i)})$ and for all $i$ and $j$, 
$$\lambda_j^{(i)} \geq \lambda_{j + c_{i+1}}^{(i+1)}, ~~\; \lambda_j^{(r)} \geq \lambda_{j+c_1}^{(1)}.$$ 
The integer $r$ is called the rank of the cylindric partition, while the integer $\ell$ is called the level of the cylindric partition.   
\end{definition}	
The size $|\Lambda|$ of a cylindric partition $\Lambda=(\lambda^{(1)}, \lambda^{(2)}, \ldots, \lambda^{(r)})$ is defined to be the sum of all the parts of the ordinary partitions included in the cylindric partition. The largest part of a cylindric partition $\Lambda$ is defined to be the maximum part among all the parts of the ordinary partitions included in the cylindric partition and denoted by $\max(\Lambda)$. For example, for the cylindric partition $\Lambda=((5,4),(8,2),(7,5,1))$ with profile $c=(1,1,1)$, $|\Lambda|=32$ and $\max(\Lambda)=8$. 
The following generating function
\begin{equation} \label{F_c(z,q)}
F_c (z, q):= \sum_{\Lambda \in P_c} z^{\max(\Lambda)} q^{|\Lambda|}
\end{equation}
is the generating function for cylindric partitions, where $P_c$ denotes the set of all cylindric partitions with profile $c$. Note that $F_c (z, q) = F_{S(c)} (z, q)$ where $c=(c_1, c_2, \ldots, c_r)$ and $S(c)=(c_2, c_3, \ldots, c_r, c_1)$ \cite{borodin2007periodic, corteel2022cylindric}. The generating function of cylindric partitions with a given profile is given by Borodin's formula \cite{borodin2007periodic}:
\begin{theorem} \label{Borodin}
Let $c=(c_1, c_2, \ldots, c_r)$ be a composition of $\ell$ where $\ell$ and $r$ are positive integers. Define $t:=r+\ell$ and $s(i,j):=c_i + c_{i+1}+ \ldots + c_j$. Then the generating function for cylindric partitions having profile $c=(c_1, c_2, \ldots, c_r)$ is given by the following formula:
\[ F_c (1, q)= \frac{1}{(q^t; q^t)_{\infty}} \prod_{i=1}^r \prod_{j=i}^r \prod_{m=1}^{c_i} \frac{1}{(q^{m+j-i+s(i+1, j)}; q^t)_{\infty}} \prod_{i=2}^r \prod_{j=2}^{i} \prod_{m=1}^{c_i} \frac{1}{(q^{t-m+j-i-s(j, i-1)}; q^t)_{\infty}}.
\]   
\end{theorem}
In the formula above and throughout, the following $q$-Pochhammer symbols \cite{gasper2004basic} are used:
\[ (a;q)_n:= \prod_{k=1}^n (1-aq^{k-1}) ~~ \text{and} ~~ (a;q)_{\infty}:=\lim_{n \rightarrow \infty} (a;q)_n ,\]
where $n \in \mathbb{N}, ~ a,q \in \mathbb{C} ~\text{and} ~ |q|<1$.

In his proof for the above formula, Borodin used a probabilistic approach \cite{borodin2007periodic}. Then, Foda and Welsh gave a proof for Borodin's formula in the context of affine and $\mathcal{W}_r$ algebras \cite{foda2016cylindric}. In their work \cite{corteel20192}, Corteel and Welsh used cylindric partitions to reproduce the four $\mathrm{A}_2$ Rogers - Ramanujan type identities and to prove a similar fifth identity. Corteel, Dousse, and Uncu worked on the cylindric partitions with a given profile for rank $3$ and level $5$, thus obtaining new $\mathrm{A}_2$ Rogers - Ramanujan type identities in \cite{corteel2022cylindric}. Warnaar's work in \cite{warnaar2023a2} gave results for $F_c(z, q)$ \eqref{F_c(z,q)} for rank $3$ and level $\not\equiv 0 \pmod{3}$. Alternative generating functions for cylindric partitions of various profiles in the form of series or infinite product times series have been considered in the works of Kanade and Russell \cite{KanadeRussell}, Uncu \cite{uncu2023proofs}, and Warnaar \cite{Warnaar2025} using other approaches. 

We study the generating functions of cylindric partitions with a profile $c=(c_1, c_2, \ldots, c_r)$ for rank $2$ and levels $2, 3$ and $4$. Borodin's formula \cite{borodin2007periodic} gives the generating functions of cylindric partitions as infinite products. In contrast to Borodin's proof of this formula, we used a more combinatorial approach in our proofs and we get alternative expressions for the generating functions of cylindric partitions with the above rank and levels. Due to the rank-level duality \cite{gessel1997cylindric}, we also get the results for cylindric partitions with rank $3$ - level $2$, and rank $4$ - level $2$. The results we have found are given below, where the left hand sides are the alternative expressions we find while the right hand sides are the infinite products obtained by Borodin's formula. The expressions we obtained for profiles $c=(1,1)$ and $c=(2,0)$ are also obtained in \cite{kurcsungoz2023combinatorial} with another combinatorial interpretation, while the expression we get for profile $c=(3,1)$ can be found in \cite{warnaar2023a2} in the context of representation theory. We have the following generating functions for cylindric partitions with the given profiles:

\noindent
Profile $c=(1,1)$:
\begin{equation} \label{c=(1,1)}
 \frac{(-q ; q^2)_{\infty}}{(q ; q)_{\infty}}=\frac{1}{(q^4;q^4)_{\infty} (q;q^4)_{\infty}^2 (q^3;q^4)_{\infty}^2}.    
\end{equation}
Profile $c=(2,0)$ :
\begin{equation} \label{c=(2,0)}
\frac{(-q^2 ; q^2)_{\infty}}{(q ; q)_{\infty}}=\frac{1}{(q;q)_{\infty} (q^2;q^4)_{\infty}}.
\end{equation}
Profiles $c=(2,1)$ and $c=(1,1,0)$: 
\begin{equation} \label{c=(2,1)}
\bigg(\sum_{n\geq 0} \frac{q^{n^2}}{(q^4; q^4)_n}\bigg) \cdot \frac{(-q^2 ; q^2)_{\infty}}{(q ; q)_{\infty}}=\frac{1}{(q;q)_{\infty} (q;q^5)_{\infty} (q^4;q^5)_{\infty}}.   
\end{equation}
Profiles $c=(3,0)$ and $c=(2,0,0)$:
\begin{equation} \label{c=(3,0)}
\bigg(\sum_{n\geq 0} \frac{q^{n^2+2 \cdot n}}{(q^4; q^4)_n}\bigg) \cdot \frac{(-q^2 ; q^2)_{\infty}}{(q ; q)_{\infty}}=\frac{1}{(q;q)_{\infty} (q^2;q^5)_{\infty} (q^3;q^5)_{\infty}}.
\end{equation}
Profiles $c=(4,0)$ and $c=(2,0,0,0)$:
\begin{equation} \label{c=(4,0)}
\bigg(\sum_{n\geq 0} \frac{q^{n(n+1)} \cdot (-q^2 ; q^2)_n }{(-q^3; q^2)_n \cdot (q^2 ; q^2)_n}\bigg) \cdot \frac{(-q^3 ; q^2)_{\infty}}{(q ; q)_{\infty}}=\frac{1}{(q;q)_{\infty} (q^2;q^6)_{\infty} (q^3;q^6)_{\infty} (q^4;q^6)_{\infty}}.
\end{equation}
Profiles $c=(2,2)$ and $c=(1,0,1,0)$:
\begin{equation} \label{c=(2,2)}
\bigg(1 + \sum_{n\geq 1} \frac{2\cdot q^{n(n+1)} \cdot (-q^2 ; q^2)_{n-1} }{(q^2; q^2)_n \cdot (-q ; q^2)_n}\bigg) \cdot \frac{(-q ; q^2)_{\infty}}{(q ; q)_{\infty}} =\frac{1}{(q^6;q^6)_{\infty}(q;q^6)_{\infty}^2 (q^2;q^6)_{\infty}^2 (q^4;q^6)_{\infty}^2 (q^5;q^6)_{\infty}^2}.
\end{equation}
Profiles $c=(3,1)$ and $c=(1,1,0,0)$:
\begin{equation} \label{c=(3,1)}
\bigg(\sum_{n\geq 0} \frac{q^{n^2}}{(q^2; q^2)_n }\bigg) \cdot \frac{(-q^2 ; q^2)_{\infty}}{(q ; q)_{\infty}}=\frac{1}{(q;q)_{\infty} (q;q^6)_{\infty} (q^3;q^6)_{\infty} (q^5;q^6)_{\infty}}. 
\end{equation}
To be able to obtain the above expressions, we fixed the profile and then studied the slices of cylindric partitions with the given profile \cite{kurcsungoz2023decomposition}. We give the definition for slices of cylindric partitions in Section \ref{sec:pre}. For any profile, there are only finitely many shapes for slices and each slice with a certain shape can have only certain weights $\pmod{r}$, where $r$ is the rank of the cylindric partition, while between weights of slices with certain shapes there are some minimum distance conditions. Therefore, we tried to implement the method of weighted words \cite{alladi1995schur, dousse2017method, dousse2017siladic} on cylindric partitions. This method was first introduced by Alladi and Gordon \cite{alladi1995schur} to prove a special (ordinary) partition identity which is Schur's identity \cite{schur1973beitrag} that states the number of partitions satisfying some congruence conditions are equal to the number of partitions satisfying some distance conditions. To prove Schur's identity, they computed the generating functions for the minimal partitions satisfying some minimal distance conditions, and used $q$-series identities. Later, this method was used by Dousse to prove other partition identities, but she created a new version of the method in which she used recurrences and $q$-difference equations instead of using minimal partitions and $q$-series identities \cite{dousse2017method, dousse2017siladic}.	
We would like to draw the reader's attention to the identities below \cite{andrews1979partitions}:
\begin{equation} \label{A1}
\sum_{n\geq 0} \frac{q^{n^2}}{(q^4; q^4)_n}=\frac{1}{(-q^2;q^2)_{\infty} (q;q^5)_{\infty} (q^4;q^5)_{\infty}}.   
\end{equation}
\begin{equation} \label{A2}
\sum_{n\geq 0} \frac{q^{n^2+2 \cdot n}}{(q^4; q^4)_n}=\frac{1}{(-q^2;q^2)_{\infty} (q^2;q^5)_{\infty} (q^3;q^5)_{\infty}}.
\end{equation}
Using \eqref{A1} and \eqref{A2}, with the expressions we obtain for profiles $c=(2,1)$ and $c=(3,0)$, we get the infinite products in \eqref{c=(2,1)} and \eqref{c=(3,0)} given by Borodin's formula \cite{borodin2007periodic}. Similarly, using identity \eqref{gasperidentity} below \cite{gasper2004basic} and Euler's famous identity for ordinary partitions, we get the infinite product in \eqref{c=(3,1)} given by Borodin's formula.
\begin{equation} \label{gasperidentity}
\sum_{n \geq 0} \frac{q^{\frac{n(n-1)}{2}}}{(q;q)_n} \cdot z^n = (-z; q)_{\infty}.
\end{equation}
We should emphasize that we were neither trying to prove an identity nor dealing with ordinary partitions. Inspired by the version of the method created by Dousse \cite{dousse2017method, dousse2017siladic}, using the minimum distance conditions and modulo congruence conditions on slices, we computed the beginning terms of the generating functions of cylindric partitions with chosen profiles. Since for a fixed profile, we deal with slices with certain shapes which can have certain weights, we did not use any total order on slices we work on or consider any dilations; but we obtained slice flows for the given profiles which allowed us to get the alternative expressions for the generating functions of cylindric partitions. Slice flows are essentially a partial order on slices \cite{kurcsungoz2023decomposition}. Instead of using recurrences and $q$-difference equations, we developed a more combinatorial way in our approach. 

For the rest of the paper, we have the following structure. In Section \ref{sec:pre}, we give necessary definitions and standard results. In Sections \ref{sec:1}, \ref{sec:2} and \ref{sec:3}, we construct the left hand sides of the formulas above. In particular, in Section \ref{sec:1} we obtain formulas \eqref{c=(1,1)}, \eqref{c=(2,0)}, in Section \ref{sec:2} we obtain formulas \eqref{c=(2,1)}, \eqref{c=(3,0)}, and finally in Section \ref{sec:3} one can find formulas \eqref{c=(4,0)}, \eqref{c=(2,2)}, \eqref{c=(3,1)}. In Section \ref{sec:future}, we give details about the future work.

\section{Preliminaries} 
\label{sec:pre}
An ordinary partition of a positive integer $n$ is a finite non-increasing sequence of positive integers $b_1, b_2, \ldots b_k$ such that $n=b_1 + b_2 + \ldots + b_k$ where the $b_i$ are called parts of the partition for $i=1,2, \ldots, k$.
If we take $n=6$, the corresponding partitions are 
\[ 6, ~~ 5+1, ~~4+2, ~~ 3+3, ~~ 4+1+1, ~~ 3+2+1, ~~ 2+2+2, ~~ 3+1+1+1, \] 
\[2+2+1+1, ~~ 2+1+1+1+1, ~~ 1+1+1+1+1+1.\]
The generating function $P(q)$ for the sequence $p(0), p(1), p(2), \ldots$ is the power series in $q$:
$$ P(q)=\sum_{k\geq0} p(k) \cdot q^k,$$
where $p(k)$ gives the number of partitions of $k$. This is called the partition function. We have the following infinite product representation of the partition function \cite{andrews1998theory}.
$$ P(q)=\prod_{k\geq1} \frac{1}{1-q^k}.$$
The last series and the infinite product converge absolutely for $|q|<1$. Considering geometric series, for any positive integer $k$ we have,
$$ \frac{1}{1-q^k}= 1 + q^k + q^{k+k}+ q^{k+k+k}+ \ldots .$$
This yields
\begin{align*}
P(q)&=\prod_{k\geq1} \frac{1}{1-q^k} \\
&=(1+q^1+q^{1+1}+ \ldots)(1+q^2+q^{2+2}+\ldots)\ldots (1+q^k+q^{k+k}+\ldots) \ldots\\
&=1+q^1+q^{1+1}+q^2+q^{1+1+1}+q^{1+2}+q^3+q^{1+1+1+1}+\ldots \\
&=1+ q^1+2q^2+3q^3+5q^4+7q^5+11q^6+ 15q^7+22q^8+ \ldots.
\end{align*}
Here, $p(0)$ is the coefficient of $q^0$ and represents the empty partition, while $p(5)=7$ is the coefficient of $q^5$ and it tells us the number of partitions of $5$ is $7$.

There are graphical representations of ordinary partitions and one them is Ferrers boards \cite{andrews2004integer}. For instance, the Ferrers board of the partition $5+3+1$ is

\ytableausetup{centertableaux}
\ytableaushort
{\none,\none }
* {5,3,1}
* [*(gray)]{0,0,0}.

If we call ordinary partitions one-dimensional partitions, we also have two-dimensional partitions, in other words, plane partitions. Plane partitions of $n$ are two-dimensional arrays of non-negative integers such that rows from left to right and columns from top to bottom are non-increasing \cite{andrews2004integer}. One of the plane partitions of $n=30$ is 
\begin{align*}
& 4 ~~ 4 ~~ 3 ~~ 2 ~~ 1 \\
& 4 ~~ 3 ~~ 1 ~~ 1\\
& 3 ~~ 2 ~~ 1\\
& 1
\end{align*}
The $3$-D Ferrers board of the above plane partition is 
\begin{center}
    \includegraphics[width=0.3\linewidth]{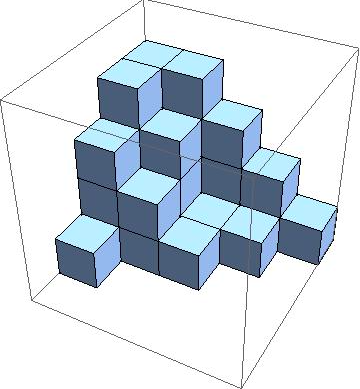}
    
{\small Figure 2.1. An example of a plane partition of $n=30$\\ (figure credit: {\url{https://en.wikipedia.org/wiki/Plane_partition}} )}
   \captionlistentry[figure]{An example of a plane partition of $n=30$}
    \label{fig:enter-label0}
\end{center}
The generating function for plane partitions is given by the following MacMahon formula \cite{macmahon2001combinatory}:
\begin{align*}
PP(q)&= \sum_{k\geq0} pp(k) \cdot q^k \\
& =\prod_{k\geq1}\frac{1}{(1-q^k)^k}.
\end{align*}
A cylindric partition with a profile $c=(c_1, c_2, \ldots, c_r)$ is a modified plane partition such that starting from the second row, rows are shifted to the left by the profile so that non-negative integers in the columns from top to bottom and in the rows from left to right are in the non-increasing form. On top of this, the first row is repeated after the last row (possibly shifted). This will be taken into account for the inequalities down each column, but it does not add to the weight. Please recall Definition \ref{def cyc} \cite{gessel1997cylindric}.

Consider the cylindric partition $\Lambda=((5,4),(8,2),(7,5,1))$ with profile $c=(1, 1, 1)$ where $\lambda^{(1)}=(5,4)$, $\lambda^{(2)}=(8,2)$, and $\lambda^{(3)}=(7,5,1)$ are ordinary partitions, with $c_1=c_2=c_3=1$ and $\lambda^{(2)}$ is shifted to the left by $c_2$, $\lambda^{(3)}$ is shifted to the left by $c_3$, and as a hidden condition $\lambda^{(1)}$ is shifted to the left by $c_1$, so that horizontally and vertically we have numbers in the non-increasing form. To get a better visualization of these shifts created by the profile, if the number of parts of ordinary partitions contained in the cylindric partition are not equal, we add zeroes to the ordinary partitions. So, for our running example, we write $\Lambda=((5,4,0),(8,2,0),(7,5,1))$ with profile $c=(1, 1, 1)$, and we do the shifts as follows: 
\[
\begin{array}{ccc ccc ccc} 
 & & & 5 & 4 & 0 \\
 && 8 & 2 & 0  & \\
 & 7 & 5 & 1  \\
\textcolor{lightgray}{5} & \textcolor{lightgray}{4} 
& \textcolor{lightgray}{0}	
\end{array}
\]
The $3$-D Ferrers board of the running example is as follows \cite{kurcsungoz2023decomposition}:
\begin{center} 
\includegraphics[width=0.3\linewidth]{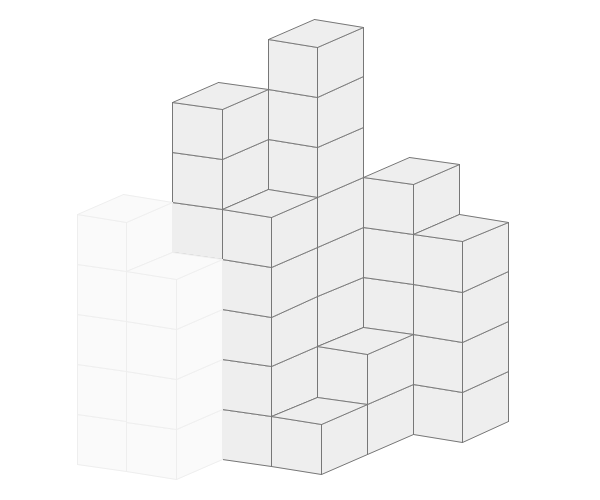}

{\small Figure 2.2. Visualization of $\Lambda=((5,4),(8,2),(7,5,1))$ with $c=(1,1,1)$}
\captionlistentry[figure]{ Visualization of $\Lambda=((5,4),(8,2),(7,5,1))$ with $c=(1,1,1)$}
 \label{fig:enter-ex1lump}
\end{center}
Note that we will work on the horizontal slices of cylindric partitions with a given profile, and we simply call them slices. We will give the formal definition after sharing the visualization for slices of the running example \cite{kurcsungoz2023decomposition}. 
\begin{center} \label{slices}
\includegraphics[width=0.4\linewidth]{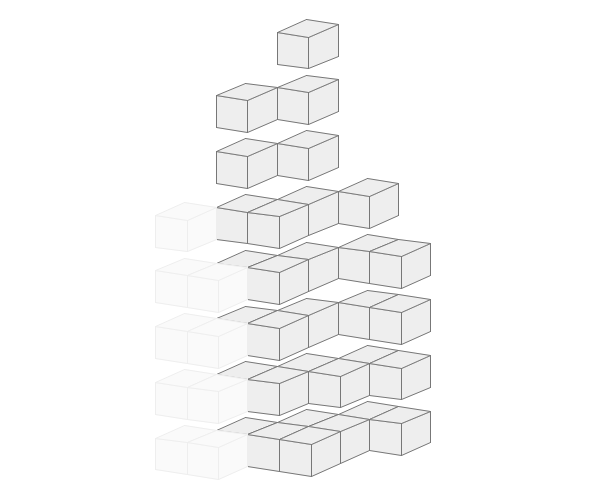}

{\small Figure 2.3. Slices of $\Lambda=((5,4),(8,2),(7,5,1))$ with $c=(1,1,1)$}
\captionlistentry[figure]{ Slices of $\Lambda=((5,4),(8,2),(7,5,1))$ with $c=(1,1,1)$}
 \label{fig:enter-slices}
\end{center}
\hfill

\begin{definition} \cite{kurcsungoz2023decomposition} \label{dfn slice}
Let $\Lambda=(\lambda^{(1)}, \lambda^{(2)}, \ldots, \lambda^{(r)})$ be a cylindric partition with profile $c=(c_1, c_2, \ldots, c_r)$ where each $\lambda^{(i)}$ is an ordinary partition such that $\lambda^{(i)}=(\lambda_1^{(i)}, \lambda_2^{(i)}, \ldots , \lambda_{s_i}^{(i)})$. Let $m:= \max (\Lambda)$, and set $\Lambda=\Lambda^{(m)}$. Then, form the cylindric partition $\prescript{}{m}\Sigma$ with the given profile so that the positions $(i,j)$ such that $\lambda_j^{(i)}=m$ are replaced by $1$'s, and all other positions are replaced by $0$'s in $\Lambda^{(m)}$. 

Subtracting $1$ from each position $(i,j)$ such that $\lambda_j^{(i)}=m$ in $\Lambda^{(m)}$, leaving all other positions as they are, we get $\Lambda^{(m-1)}$. Now, form the cylindric partition $\prescript{}{m-1}\Sigma$ with the given profile so that the positions $(i,j)$ such that $\lambda_j^{(i)}=m-1$ are replaced by $1$'s, and all other positions are replaced by $0$'s in $\Lambda^{(m-1)}$.

We continue in this fashion until we obtain $\max \lambda_j^{(i)}=0$, in other words, until we obtain the empty partition with the given profile. We denote it by $\prescript{}{0}{} \Sigma$.

These $\prescript{}{i} \Sigma$'s are called slices of cylindric partition $\Lambda$ with profile $c=(c_1, c_2, \ldots, c_r)$.  
\end{definition}
Now, let us share with you the slices of the running example $\Lambda=((5,4),(8,2),(7,5,1))$ with profile $c=(1, 1, 1)$, starting with the empty slice $\prescript{}{0}{} \Sigma$. Note that gray squares are created by the profile and have zero weight, while white squares have weight $1$. Except the gray squares created by the profile, other squares having zero weight are not represented by gray squares. We have the following order for the slices of the running example: $\prescript{}{0}{\Sigma} \leq \prescript{}{8}{\Sigma} \leq \prescript{}{7}{\Sigma} \leq \ldots \leq \prescript{}{1}{\Sigma}$. Their visual representations are given below.
\hfill

\ytableausetup{centertableaux}
\ytableaushort
{\none,\none }
* {3,2,1}
* [*(gray)]{3,2,1} ~ $\leq$ 
\ytableausetup{centertableaux}
\ytableaushort
{\none,\none }
* {3,3,1}
* [*(gray)]{3,2,1} ~ $\leq$
\ytableausetup{centertableaux}
\ytableaushort
{\none,\none }
* {3,3,2}
* [*(gray)]{3,2,1} ~ $\leq$ 
\ytableausetup{centertableaux}
\ytableaushort
{\none,\none }
* {3,3,2}
* [*(gray)]{3,2,1} ~ $\leq$ 
\ytableausetup{centertableaux}
\ytableaushort
{\none,\none }
* {4,3,3}
* [*(gray)]{3,2,1} ~ $\leq$ 
\hfill

\ytableausetup{centertableaux}
\ytableaushort
{\none,\none }
* {5,3,3}
* [*(gray)]{3,2,1} ~ $\leq$ 
\ytableausetup{centertableaux}
\ytableaushort
{\none,\none }
* {5,3,3}
* [*(gray)]{3,2,1} ~ $\leq$ 
\ytableausetup{centertableaux}
\ytableaushort
{\none,\none }
* {5,4,3}
* [*(gray)]{3,2,1} ~ $\leq$ 
\hfill

\ytableausetup{centertableaux}
\ytableaushort
{\none,\none }
* {5,4,4}
* [*(gray)]{3,2,1} ~. 
\hfill

Now, if we put the slices in the order above, so that the gray squares overlap and $\prescript{}{8}{\Sigma}$ is on the top, the total number of white squares in the same position gives the weight in that position, and we get the cylindric partition $\Lambda=((5,4),(8,2),(7,5,1))$ with profile $c=(1, 1, 1)$. Note that this is another visualization for Figure $2.3$, and we will use it in our proofs.  

Each slice having a profile $c=(c_1, c_2, \ldots, c_r)$ has a shape. We define the shape of a slice as follows: Without distinguishing between the white and gray squares in the representation of a slice, we remove $k$ squares from each row, where $k$ is the length of the $r^{th}$ row. Since in a slice with a profile $c=(c_1, c_2, \ldots, c_r)$, there are $r$ rows, after erasing the $r^{th}$ row, we will be left with $(r-1)$ rows (counting zeroes), and starting from the top row, we write the number of squares left in the rows, respectively. We call this the shape of a slice \cite{kurcsungoz2023decomposition}. Shapes of $\prescript{}{0}{\Sigma}, \prescript{}{8}{\Sigma}, \prescript{}{7}{\Sigma}, \ldots, \prescript{}{1}{\Sigma}$ above are $(2,1), ~ (2,2), ~ (1,1), ~ (1,1), ~ (1,0), ~ (2,0), ~ (2,0), ~ (2,1)$ and $(1,0)$, respectively.

Note that for a fixed profile, there are exactly 
$\binom{\ell+r-1}{r-1}$ shapes, 
and each shape has a certain weight $\pmod{r}$ \cite{kurcsungoz2023decomposition}. Starting with the minimal slice with a certain shape, we get other slices with the same shape by adding the same number of white squares (with weight $1$) to each row. In the small example below, one can follow this process.

\ytableausetup{centertableaux}
\ytableaushort
{\none,\none }
* {3,3,1}
* [*(gray)]{3,2,1} ~ $\leq$
\ytableausetup{centertableaux}
\ytableaushort
{\none,\none }
* {4,4,2}
* [*(gray)]{3,2,1} ~ $\leq$
\ytableausetup{centertableaux}
\ytableaushort
{\none,\none }
* {5,5,3}
* [*(gray)]{3,2,1} ~.

After giving the necessary information related to the theory, it is time to explain how we apply the method. First of all, we fix the profile. We determine the number of shapes and the corresponding minimal slices having positive weight, and then we determine the minimum distance conditions between the slices. We construct cylindric partitions with a given profile by starting with the slice on the top. After the slice on the top, we can put another slice that contains every single white square in the representation of the slice on the top. After the second slice, if we choose a third slice, it must contain every single white square in the representation of the second slice. And we continue in this fashion.

Consider cylindric partitions having profile $c = (2, 1)$ with rank $r=2$ and level $\ell=3$. So for the slices of the cylindric partitions with the given profile, there are exactly 
$ \binom{\ell+r-1}{r-1}= $ $\binom{4}{1}$
shapes. The corresponding shapes are $(0), (1), (2), (3)$ and let us denote these shapes by $\mathbf{a}, \mathbf{c}, \mathbf{b}$, and $\mathbf{d}$, respectively. Slices that have these shapes with minimum positive weight are shown below. Note that the colored squares are created by the profile and have zero weight. Different colors are used to distinguish the different shapes. Although $aq^1$, $bq^1$, $cq^2$, $dq^2$ are the generating terms of the slices given below, for simplicity of the language, we will name the slices with their generating terms. 

\ytableausetup{centertableaux}
\ytableaushort
{\none,\none }
* {3,3}
* [*(yellow)]{3,2}  
~ ($aq^1$), ~ \ytableausetup{centertableaux}
\ytableaushort
{\none,\none }
* {4,2}
* [*(pink)]{3,2}
~ ($bq^1$),
\ytableausetup{centertableaux}
\ytableaushort
{\none,\none }
* {4,3}
* [*(blue)]{3,2}  
~ ($cq^2$),

\ytableausetup{centertableaux}
\ytableaushort
{\none,\none }
* {5,2}
* [*(red)]{3,2} ~ ($dq^2$).

For example, as shown below, after the slice $aq^1$ we cannot put the slice $dq^2$, as the crossed white square in the slice $aq^1$ is not contained in $dq^2$.  
\begin{center} 
\includegraphics[width=0.25\linewidth]{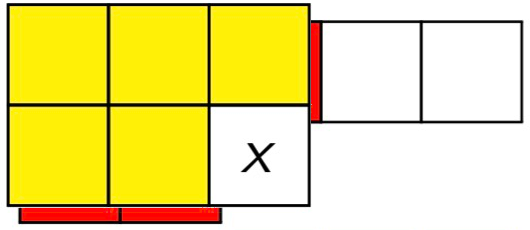}

{\small Figure 2.4. A non-example for slice placement}
\captionlistentry[figure]{ A non-example for slice placement}
 \label{fig:enter-label}
\end{center}
But after $aq^1$, we can put $dq^4$. Moreover, we can have the following placement with the slices $aq^1, bq^3$ and $dq^4$, respectively. 
\begin{center} 
\includegraphics[width=0.3\linewidth]{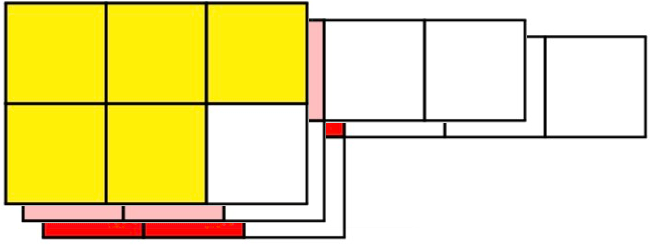}

{\small Figure 2.5. An example for slice placement}
\captionlistentry[figure]{ An example for slice placement}
 \label{fig:enter-slide1}
\end{center}
The above placement of the slices creates the cylindric partition $\Lambda=((2, 2, 1), (3))$ with profile $c=(2, 1)$.  

In the following sections, using the minimum distance conditions on slices, we create slice flows for the given profiles that give the pattern of slices with distinct shapes including weights. 
\begin{definition}
Given an arbitrary but fixed profile $c$,
the \emph{slice flow} for the given profile is a directed graph
in which the generating term $x q^n$ for each non-empty slice is a node,
and there is an edge from $x q^n$ to $y q^{n+1}$ if and only if the slice generated by $y q^{n+1}$
can be obtained from the slice generated by $x q^n$ by adding a single square.
\end{definition}
Slice flow for a given profile is a diagram that gives the partial order on the slices with positive weight for the given profile, where the slices are denoted by their generating terms. Slice flow for a given profile is another form of Hasse diagram in \cite{kurcsungoz2023decomposition}, but instead of using visualizations of the slices with a given profile, we denote the shapes of the slices by words ($a, ~ b$, etc.) and by using the arrows we show the placement rules for slices with the given profile. We did not include the empty slices in the slice flows. For convenience, we horizontally aligned slices with the same shape,
and vertically aligned slices with the same weight.  
The weights increase from left to right. Let us share with you the following slice flow for profile $c=(2, 1)$: 
\footnote{Latex codes for the slice flows are created by the tikzcd editor \url{https://tikzcd.yichuanshen.de}.}
\begin{center}
    \begin{tikzcd}
aq^1 \arrow[rd]            &                            & aq^3 \arrow[rd]            &                            & aq^5 \arrow[rd]            &                            & aq^7 \arrow[rd]            &      \\
                           & cq^2 \arrow[ru] \arrow[rd] &                            & cq^4 \arrow[ru] \arrow[rd] &                            & cq^6 \arrow[rd] \arrow[ru] &                            & cq^8 \\
bq^1 \arrow[ru] \arrow[rd] &                            & bq^3 \arrow[ru] \arrow[rd] &                            & bq^5 \arrow[ru] \arrow[rd] &                            & bq^7 \arrow[ru] \arrow[rd] &      \\
                           & dq^2 \arrow[ru]            &                            & dq^4 \arrow[ru]            &                            & dq^6 \arrow[ru]            &                            & dq^8
\end{tikzcd} \\
 {\small Table 2.1: Profile $c=(2,1)$}
 \captionlistentry[table]{Profile $c=(2,1)$}
\end{center}

We use the slice flows to produce the generating functions of cylindric partitions for given profiles and this process will be explained in the following sections. 
	
\section{Method of Weighted Words on Cylindric Partitions 
with Rank 2 - Level 2}
\label{sec:1}

We start our work applying the method of weighted words on cylindric partitions with rank $2$ - level $2$. There are only two profiles in this case, profile $c=(1,1)$ and $c=(2,0)$.   

\begin{theorem} \label{theorem5}
Cylindric partitions with profile $c=(1,1)$ have the following generating function:
$$ \frac{(-q ; q^2)_{\infty}}{(q ; q)_{\infty}}. $$
\end{theorem}

\begin{proof}
For cylindric partitions with profile $c = (1, 1)$, there are exactly $ \binom{\ell+r-1}{r-1}= $ $\binom{3}{1}$ shapes. The corresponding shapes are $(0), (1), (2)$ and we will denote these shapes by $a, c$ and $b$, respectively. Slices that have these shapes with minimum positive weight are shown below. Note that the gray squares are created by the profile and have zero weight.

\ytableausetup{centertableaux}
\ytableaushort
{\none,\none }
* {2,2}
* [*(gray)]{2,1} ~ ($aq^1$), \ytableausetup{centertableaux}
\ytableaushort
{\none,\none }
* {3,1}
* [*(gray)]{2,1}
~ ($bq^1$),
\ytableausetup{centertableaux}
\ytableaushort
{\none,\none }
* {3,2}
* [*(gray)]{2,1}  
~ ($cq^2$).

Since the bigger slices with a certain shape are obtained by adding the same number of weight $1$ squares to each row of the minimal slice with the given shape and the number of rows in a slice is determined by the rank, each shape has certain weights modulo $2$ for the given profile. Therefore, slices with shapes $a$ and $b$ can only have odd weights, while slices with shape $c$ can only have even weights. Slice flow for the given profile is:

\begin{center}
\begin{tikzcd}
\color{red}{aq^1} \arrow[rd] &                                         & \color{red}{aq^3} \arrow[rd] &                                         & \color{red}{aq^5} \arrow[rd] &                                         & \color{red}{aq^7} \arrow[rd] &                   \\
                             & \color{red}{cq^2} \arrow[rd] \arrow[ru] &                              & \color{red}{cq^4} \arrow[rd] \arrow[ru] &                              & \color{red}{cq^6} \arrow[rd] \arrow[ru] &                              & \color{red}{cq^8} \\
\color{red}{bq^1} \arrow[ru] &                                         & \color{red}{bq^3} \arrow[ru] &                                         & \color{red}{bq^5} \arrow[ru] &                                         & \color{red}{bq^7} \arrow[ru] &                  
\end{tikzcd}\\
 {\small Table 3.1: Profile $c=(1,1)$}
 \captionlistentry[table]{Profile $c=(1,1)$}
\end{center}

It is straightforward to write the generating function of the cylindric partitions with distinct slices:
$$\prod_{k \geq 0} (1 + aq^{2k+1}+bq^{2k+1}) (1 + cq^{2k+2}).$$
Allowing slices to repeat, we obtain
$$\prod_{k \geq 0} \bigg(1 + \frac{aq^{2k+1}}{1-aq^{2k+1}}+ \frac{bq^{2k+1}}{1-bq^{2k+1}}\bigg) \bigg(1 + \frac{cq^{2k+2}}{1-cq^{2k+2}}\bigg)$$
Now, if we put $a=b=c=1$,this concludes the proof.
\end{proof}

\begin{theorem} \label{theorem6}
Cylindric partitions with profile $c=(2,0)$ have the following generating function:
$$ \frac{(-q^2 ; q^2)_{\infty}}{(q ; q)_{\infty}}. $$
\end{theorem}

\begin{proof}
For cylindric partitions with profile $c = (2, 0)$, there are exactly $ \binom{\ell+r-1}{r-1}= $ $\binom{3}{1}$ shapes. The corresponding shapes are again $(0), (1), (2)$ and we will denote these shapes by $b, a$ and $c$, respectively. Slices that have these shapes with minimum positive weight are shown below. 

\ytableausetup{centertableaux}
\ytableaushort
{\none,\none }
* {3,2}
* [*(gray)]{2,2} ~ ($aq^1$), \ytableausetup{centertableaux}
\ytableaushort
{\none,\none }
* {3,3}
* [*(gray)]{2,2}
~ ($bq^2$),
\ytableausetup{centertableaux}
\ytableaushort
{\none,\none }
* {4,2}
* [*(gray)]{2,2}  
~ ($cq^2$).

Slices with shape $a$ can only have odd weights, while slices with shapes $b$ and $c$ can only have even weights. Slice flow for the profile can be found below.
\begin{center}
\begin{tikzcd}
                                        & \color{red}{cq^2} \arrow[rd] &                                         & \color{red}{cq^4} \arrow[rd] &                                         & \color{red}{cq^6} \arrow[rd] &                                         & \color{red}{cq^8} \\
\color{red}{aq^1} \arrow[ru] \arrow[rd] &                              & \color{red}{aq^3} \arrow[ru] \arrow[rd] &                              & \color{red}{aq^5} \arrow[ru] \arrow[rd] &                              & \color{red}{aq^7} \arrow[ru] \arrow[rd] &                   \\
                                        & \color{red}{bq^2} \arrow[ru] &                                         & \color{red}{bq^4} \arrow[ru] &                                         & \color{red}{bq^6} \arrow[ru] &                                         & \color{red}{bq^8}
\end{tikzcd}\\
 {\small Table 3.2: Profile $c=(2,0)$}
 \captionlistentry[table]{Profile $c=(2,0)$}
\end{center}
As in the previous theorem, it is easy to write the generating function of the cylindric partitions with distinct slices:
$$\prod_{k \geq 0} (1 + aq^{2k+1}) (1 + bq^{2k+2}+cq^{2k+2}).$$
Now, if we allow slices to repeat, we get
$$\prod_{k \geq 0} \bigg(1 + \frac{aq^{2k+1}}{1-aq^{2k+1}}\bigg) \bigg(1 + \frac{bq^{2k+2}}{1-bq^{2k+2}}+\frac{cq^{2k+2}}{1-cq^{2k+2}}\bigg).$$
Putting $a=b=c=1$ completes the proof.
\end{proof}

Note that in the above proofs, even if we do not put $a=b=c=1$, we still have infinite products, which also count the number of shapes of slices included in cylindric partitions with the given profiles via the powers of $a, b$ and $c$. In the following section, for the profiles we worked on, this does not happen, as we deal with more complex slice flows.  

\section{Method of Weighted Words on Cylindric Partitions with Rank 2 - Level 3 and Cylindric Partitions with Rank 3 - Level 2}
\label{sec:2}
In this section, we apply the method of weighted words on cylindric partitions with rank $2$ - level $3$, and rank $3$ - level $2$.  

We start our examination with cylindric partitions having profile $c=(2,1)$, and we find an alternative expression for the generating function of cylindric partitions with the given profile.  

\begin{theorem} \label{theorem1}
Cylindric partitions with profile $c=(2,1)$ have the following generating function:
$$\bigg(\sum_{n\geq 0} \frac{q^{n^2}}{(q^4 ; q^4)_n}\bigg) \cdot \frac{(-q^2 ; q^2)_{\infty}}{(q ; q)_{\infty}}. $$
\end{theorem}

\begin{proof}
Cylindric partitions with profile $c = (2, 1)$ have rank $r=2$ and level $\ell=3$, so for the slices of cylindric partitions with the given profile there are exactly $ \binom{\ell+r-1}{r-1}= $ $\binom{4}{1}$ shapes. The corresponding shapes are $(0), (1), (2), (3)$ and we will denote these shapes by $a, c, b$, and $d$, respectively. Slices that have these shapes with minimum positive weight are shown below. Note that the gray squares are created by the profile and have zero weight.

\ytableausetup{centertableaux}
\ytableaushort
{\none,\none }
* {3,3}
* [*(gray)]{3,2} ~ ($aq^1$), \ytableausetup{centertableaux}
\ytableaushort
{\none,\none }
* {4,2}
* [*(gray)]{3,2}
~ ($bq^1$),
\ytableausetup{centertableaux}
\ytableaushort
{\none,\none }
* {4,3}
* [*(gray)]{3,2}  
~ ($cq^2$),

\hfill

\ytableausetup{centertableaux}
\ytableaushort
{\none,\none }
* {5,2}
* [*(gray)]{3,2} ~ ($dq^2$).

As the bigger slices with a certain shape are obtained by adding the same number of weight $1$ squares to each row of the minimal slice with the given shape and the number of rows in a slice is determined by the rank, each shape has certain weights modulo $2$ for the given profile. Therefore, slices with shapes $a$ and $b$ can only have odd weights, while slices with shapes $c$ and $d$ can only have even weights. Slice flow for the given profile can be found below:
\begin{center}
    \begin{tikzcd}
\color{blue}{aq^1} \arrow[rd]           &                                         & \color{blue}{aq^3} \arrow[rd]           &                                         & \color{blue}{aq^5} \arrow[rd]           &                                         & \color{blue}{aq^7} \arrow[rd]           &                   \\
                                        & \color{red}{cq^2} \arrow[ru] \arrow[rd] &                                         & \color{red}{cq^4} \arrow[ru] \arrow[rd] &                                         & \color{red}{cq^6} \arrow[rd] \arrow[ru] &                                         & \color{red}{cq^8} \\
\color{red}{bq^1} \arrow[ru] \arrow[rd] &                                         & \color{red}{bq^3} \arrow[ru] \arrow[rd] &                                         & \color{red}{bq^5} \arrow[ru] \arrow[rd] &                                         & \color{red}{bq^7} \arrow[ru] \arrow[rd] &                   \\
                                        & \color{red}{dq^2} \arrow[ru]            &                                         & \color{red}{dq^4} \arrow[ru]            &                                         & \color{red}{dq^6} \arrow[ru]            &                                         & \color{red}{dq^8}
\end{tikzcd}\\
 {\small Table 4.1: Profile $c=(2,1)$}
 \captionlistentry[table]{Profile $c=(2,1)$}
\end{center}

Investigating the slices for the given profile, it's straightforward to find that cylindric partitions with distinct slices including only the slices with shapes $b$, $c$ and $d$ (all shapes except $a$) have the following generating function:
$$\prod_{k \geq 0} (1 + bq^{2k+1}) (1 + cq^{2k+2} + dq^{2k+2}).$$

Whenever a slice with shape $a$ enters the picture, we know that we must be careful with the slices with shape $d$, since just before and just after a slice $aq^{2k+1}$, we cannot put slices $dq^{2k}$ and $dq^{2k+2}$. Considering this restriction, the possible arrangements of distinct slices with maximum weight $6$ ($cq^6$ or $dq^6$) are as follows:

$(1+bq^1)(1+cq^2+dq^2)(1+bq^3)(1+cq^4+dq^4)(1+bq^5)(1+cq^6+dq^6)$

$+ \mathbf{(aq^1)(1+cq^2)}(1+bq^3)(1+cq^4+dq^4)(1+bq^5)(1+cq^6+dq^6) $ 

$+ (1+bq^1)\mathbf{(1+cq^2)(aq^3)(1+cq^4)}(1+bq^5)(1+cq^6+dq^6) $

$+ (1+bq^1)(1+cq^2+dq^2)(1+bq^3)\mathbf{(1+cq^4)(aq^5)(1+cq^6)} $ 

$+ \mathbf{(aq^1)(1+cq^2)(aq^3)(1+cq^4)}(1+bq^5)(1+cq^6+dq^6) $

$+ \mathbf{(aq^1)(1+cq^2)}(1+bq^3)\mathbf{(1+cq^4)(aq^5)(1+cq^6)} $

$+ (1+bq^1)\mathbf{(1+cq^2)(aq^3)(1+cq^4)(aq^5)(1+cq^6)} $ 

$+ \mathbf{(aq^1)(1+cq^2)(aq^3)(1+cq^4)(aq^5)(1+cq^6)}$. 

Since the maximum weight is restricted by $6$ above, the possible weights of slices with shape $a$ are $1, 3,$ and $5$. There is a one-to-one correspondence between the subsets of the set $\{ aq^1, aq^3, aq^5\}$, and the terms of the sum above, each line representing one term. For example, the first line of the sum corresponds to the empty set so that no slice with shape $a$ is included, while the second line of the sum corresponds to the subset $\{ aq^1\}$ so that the slice $aq^1$ is included, and the seventh line of the sum corresponds to the subset $\{ aq^3, aq^5\}$ so that slices $aq^3$ and $aq^5$ are included. Therefore, the sum above has eight terms. Now, if we allow slices to repeat, this sum becomes:

$(1+\frac{bq^1}{1-bq^1})(1+\frac{cq^2}{1-cq^2}+\frac{dq^2}{1-dq^2})(1+\frac{bq^3}{1-bq^3})(1+\frac{cq^4}{1-cq^4}+\frac{dq^4}{1-dq^4})(1+\frac{bq^5}{1-bq^5})(1+\frac{cq^6}{1-cq^6}+\frac{dq^6}{1-dq^6})$

$+ \mathbf{(\frac{aq^1}{1-aq^1})(1+\frac{cq^2}{1-cq^2})}(1+\frac{bq^3}{1-bq^3})(1+\frac{cq^4}{1-cq^4}+\frac{dq^4}{1-dq^4})(1+\frac{bq^5}{1-bq^5})(1+\frac{cq^6}{1-cq^6}+\frac{dq^6}{1-dq^6}) $ 

$+ (1+\frac{bq^1}{1-bq^1})\mathbf{(1+\frac{cq^2}{1-cq^2})(\frac{aq^3}{1-aq^3})(1+\frac{cq^4}{1-cq^4})}(1+\frac{bq^5}{1-bq^5})(1+\frac{cq^6}{1-cq^6}+\frac{dq^6}{1-dq^6}) $

$+ (1+\frac{bq^1}{1-bq^1})(1+\frac{cq^2}{1-cq^2}+\frac{dq^2}{1-dq^2})(1+\frac{bq^3}{1-bq^3})\mathbf{(1+\frac{cq^4}{1-cq^4})(\frac{aq^5}{1-aq^5})(1+\frac{cq^6}{1-cq^6})} $ 

$+ \mathbf{(\frac{aq^1}{1-aq^1})(1+\frac{cq^2}{1-cq^2})(\frac{aq^3}{1-aq^3})(1+\frac{cq^4}{1-cq^4})}(1+\frac{bq^5}{1-bq^5})(1+\frac{cq^6}{1-cq^6}+\frac{dq^6}{1-dq^6}) $

$+ \mathbf{(\frac{aq^1}{1-aq^1})(1+\frac{cq^2}{1-cq^2})}(1+\frac{bq^3}{1-bq^3})\mathbf{(1+\frac{cq^4}{1-cq^4})(\frac{aq^5}{1-aq^5})(1+\frac{cq^6}{1-cq^6})} $

$+ (1+\frac{bq^1}{1-bq^1})\mathbf{(1+\frac{cq^2}{1-cq^2})(\frac{aq^3}{1-aq^3})(1+\frac{cq^4}{1-cq^4})(\frac{aq^5}{1-aq^5})(1+\frac{cq^6}{1-cq^6})} $ 

$+ \mathbf{(\frac{aq^1}{1-aq^1})(1+\frac{cq^2}{1-cq^2})(\frac{aq^3}{1-aq^3})(1+\frac{cq^4}{1-cq^4})(\frac{aq^5}{1-aq^5})(1+\frac{cq^6}{1-cq^6})}$. 

Then, if we put $a=b=c=d=1$, we obtain the following sum:

$(\frac{1}{1-q^1})(\frac{1 + q^2}{1 - q^2})(\frac{1}{1 - q^3})(\frac{1+q^4}{1-q^4})(\frac{1}{1-q^5})(\frac{1+q^6}{1-q^6})$ 

$+ \mathbf{(\frac{q^1}{1-q^1})(\frac{1}{1-q^2})}(\frac{1}{1-q^3})(\frac{1+q^4}{1-q^4})(\frac{1}{1-q^5})(\frac{1+q^6}{1-q^6}) $ 

$+ (\frac{1}{1-q^1})\mathbf{(\frac{1}{1-q^2})(\frac{q^3}{1-q^3})(\frac{1}{1-q^4})}(\frac{1}{1-q^5})(\frac{1+q^6}{1-q^6}) $ 

$+ (\frac{1}{1-q^1})(\frac{1+q^2}{1-q^2})(\frac{1}{1-q^3})\mathbf{\frac{1}{1-q^4}(\frac{q^5}{1-q^5})(\frac{1}{1-q^6})} $ 

$+ \mathbf{(\frac{q^1}{1-q^1})(\frac{1}{1-q^2})(\frac{q^3}{1-q^3})(\frac{1}{1-q^4})}(\frac{1}{1-q^5})(\frac{1+q^6}{1-q^6}) $ 

$+ \mathbf{(\frac{q^1}{1-q^1})(\frac{1}{1-q^2})}(\frac{1}{1-q^3})\mathbf{(\frac{1}{1-q^4})(\frac{q^5}{1-q^5})(\frac{1}{1-q^6})} $ 

$+ (\frac{1}{1-q^1})\mathbf{(\frac{1}{1-q^2})(\frac{q^3}{1-q^3})(\frac{1}{1-q^4})(\frac{q^5}{1-q^5})(\frac{1}{1-q^6})} $ 

$+ \mathbf{(\frac{q^1}{1-q^1})(\frac{1}{1-q^2})(\frac{q^3}{1-q^3})(\frac{1}{1-q^4})(\frac{q^5}{1-q^5})(\frac{1}{1-q^6})}$. 

Let us denote the first term of this new sum above by $P(q)$, that is,
$$P(q):= (\frac{1}{1-q^1})(\frac{1 + q^2}{1 - q^2})(\frac{1}{1 - q^3})(\frac{1+q^4}{1-q^4})(\frac{1}{1-q^5})(\frac{1+q^6}{1-q^6}).$$ 
Then, how we get the other terms of the sum is as follows:

$\mathbf{(\frac{q^1}{1-q^1})(\frac{1}{1-q^2})}(\frac{1}{1-q^3})(\frac{1+q^4}{1-q^4})(\frac{1}{1-q^5})(\frac{1+q^6}{1-q^6}) = P(q)\cdot \frac{q^1}{1+q^2} $ 

$(\frac{1}{1-q^1})\mathbf{(\frac{1}{1-q^2})(\frac{q^3}{1-q^3})(\frac{1}{1-q^4})}(\frac{1}{1-q^5})(\frac{1+q^6}{1-q^6})= P(q)\cdot \frac{q^3}{(1+q^2)(1+q^4)}$ 

$(\frac{1}{1-q^1})(\frac{1+q^2}{1-q^2})(\frac{1}{1-q^3})\mathbf{\frac{1}{1-q^4}(\frac{q^5}{1-q^5})(\frac{1}{1-q^6})}= P(q)\cdot \frac{q^5}{(1+q^4)(1+q^6)}$ 

$\mathbf{(\frac{q^1}{1-q^1})(\frac{1}{1-q^2})(\frac{q^3}{1-q^3})(\frac{1}{1-q^4})}(\frac{1}{1-q^5})(\frac{1+q^6}{1-q^6})= P(q) \cdot \frac{q^1 \cdot q^3}{(1+q^2)(1+q^4)}$ 

$\mathbf{(\frac{q^1}{1-q^1})(\frac{1}{1-q^2})}(\frac{1}{1-q^3})\mathbf{(\frac{1}{1-q^4})(\frac{q^5}{1-q^5})(\frac{1}{1-q^6})}= P(q) \cdot \frac{q^1 \cdot q^5}{(1+q^2)(1+q^4)(1+q^6)}$ 

$(\frac{1}{1-q^1})\mathbf{(\frac{1}{1-q^2})(\frac{q^3}{1-q^3})(\frac{1}{1-q^4})(\frac{q^5}{1-q^5})(\frac{1}{1-q^6})}= P(q) \cdot \frac{q^3 \cdot q^5}{(1+q^2)(1+q^4)(1+q^6)}$

$\mathbf{(\frac{q^1}{1-q^1})(\frac{1}{1-q^2})(\frac{q^3}{1-q^3})(\frac{1}{1-q^4})(\frac{q^5}{1-q^5})(\frac{1}{1-q^6})}= P(q) \cdot \frac{q^1 \cdot q^3 \cdot q^5}{(1+q^2)(1+q^4)(1+q^6)}$. 

Now, we will generalize this idea.

Cylindric partitions with the given profile not including a slice with shape $a$ have the following generating function:
\begin{align} \label{n=0}
S(q):= \frac{\prod_{k \geq 1} (1 + q^{2k})}{\prod_{k \geq 1} (1-q^k) }= \frac{(-q^2;q^2)_{\infty}}{(q;q)_{\infty}}. 
\end{align}
Cylindric partitions with the given profile including only one size of slice with shape $a$, that is, including $aq^{2k+1}$ for some $k \in \mathbb{N}$, have the following generating function:
\begin{align} \label{n=1}
\bigg( \frac{q^1}{1+q^2} + \sum_{k\geq1} \frac {q^{2k+1}}{(1+q^{2k}) (1+q^{2k+2})}\bigg)\cdot S(q). 
\end{align}
Cylindric partitions with the given profile including exactly two sizes of slices with shape $a$ have the following generating function:
\begin{align} \label{n=2}
&\bigg(\frac{q^1 \cdot q^3}{(1+q^2) (1+q^4)} + \sum_{k\geq1} \frac {q^{1} \cdot q^{2k+3}}{ (1+q^2) (1+q^{2k+2}) (1+q^{2k+4})} + \sum_{k\geq1} \frac {q^{2k+1} \cdot q^{2k+3}}{ (1+q^{2k}) (1+q^{2k+2}) (1+q^{2k+4})} \\
\nonumber
&+ \sum_{n\geq1}\sum_{k\geq1} \frac {q^{2k+1} \cdot q^{2k+2n+3}}{ (1+q^{2k}) (1+q^{2k+2}) (1+q^{2k+2n+2}) (1+q^{2k+2n+4})} \bigg) \cdot S(q).
\end{align}

Cylindric partitions with the given profile including exactly three sizes of slices with shape $a$ have the following generating function:

\begin{align} \label{n=3}
&S(q) \cdot \bigg(\frac{q^{1}\cdot q^{3} \cdot q^{5}}{(1+q^2) (1+q^4) (1+q^6)} + 
\sum_{k\geq 1}^{\infty}\frac{q^{1} \cdot q^{3} \cdot q^{2k+5}}{(1+q^2) (1+q^4) (1+q^{2k+4}) (1+q^{2k+6})} \\
\nonumber
&+ \sum_{k\geq1} \frac {q^{1} \cdot q^{2k+3} \cdot q^{2k+5}}{ (1+q^2) (1+q^{2k+2}) (1+q^{2k+4}) (1+q^{2k+6})} \\
\nonumber
& + \sum_{n\geq 1} \sum_{k\geq1} \frac {q^{1} \cdot q^{2k+3} \cdot q^{2k+2n+5}}{ (1+q^2) (1+q^{2k+2}) (1+q^{2k+4}) (1+q^{2k+2n+4}) (1+q^{2k+2n+6})} \\
\nonumber
& + \sum_{k\geq1} \frac {q^{2k+1} \cdot q^{2k+3} \cdot q^{2k+5}}{ (1+q^{2k}) (1+q^{2k+2}) (1+q^{2k+4}) (1+q^{2k+6})} \\
\nonumber
& + \sum_{k\geq1} \frac {q^{2k+1} \cdot q^{2k+3} \cdot q^{2k+2n+5} }{ (1+q^{2k}) (1+q^{2k+2}) (1+q^{2k+4}) (1+q^{2k+2n+4}) (1+q^{2k+2n+6})} 
\end{align}

\begin{align}
\nonumber
& + \sum_{n\geq1}\sum_{k\geq1} \frac {q^{2k+1} \cdot q^{2k+2n+3} \cdot q^{2k+2n+5}}{ (1+q^{2k}) (1+q^{2k+2}) (1+q^{2k+2n+2}) (1+q^{2k+2n+4}) (1+q^{2k+2n+6})} \\
\nonumber
& + \sum_{m, n, k\geq1} \frac {q^{2k+1} \cdot q^{2k+2n+3} \cdot q^{2k+2n+2m+5}}{ (1+q^{2k}) (1+q^{2k+2}) (1+q^{2k+2n+2}) (1+q^{2k+2n+4}) (1+q^{2k+2n+2m+4}) (1+q^{2k+2n+2m+6})} \bigg). 
\end{align}
As the number of distinct sizes of slices with shape $a$ increases, our expression gets larger and larger and more complicated. Fortunately, there is a certain pattern behind the scenes. Now, let us prove the following lemmas which will help us to compute the above sums:

\begin{lemma} \label{lem1}
For any integer $k \geq 0$, we have the following identity:
\begin{equation} \label{eq1}
\frac{q^{2k+1}}{(1+q^{2k}) (1+q^{2k+2})}=\frac{q}{1-q^2} \cdot \bigg( \frac{-1}{1+q^{2k}} + \frac{1}{1+ q^{2k+2}} \bigg).    
\end{equation}
For any integers $k \geq 0$, $n \geq 2$, we have the following identity,
\begin{multline} \label{eq2}
\frac{q^{2k+1} \cdot q^{2k+3} \ldots q^{2k+2n-1}}{(1+q^{2k}) (1+q^{2k+2}) \ldots (1+q^{2k+2n})}=  \\
\frac{q}{1-q^{2n}} \bigg( \frac{-q^{2k+3}\cdot q^{2k+5} \ldots q^{2k+2n-1}}{(1+q^{2k})(1+q^{2k+2})\ldots(1+q^{2k+2n-2})} +
\frac{q^{2k+3}\cdot q^{2k+5} \ldots q^{2k+2n-1} }{(1+q^{2k+2})(1+q^{2k+4})\ldots(1+q^{2k+2n})} \bigg).
\end{multline}
\end{lemma}

\begin{proof}
Identity \eqref{eq1} easily can be verified. For Identity \eqref{eq2} we will proceed by induction on $n$. The base case $n=2$ can be seen by verification. Now, supposing Identity \eqref{eq2} holds for some $n \geq 2$, we will prove it for $n+1$, that is, we will prove the following:
\begin{multline} \label{eq3}
\frac{q^{2k+1} \cdot q^{2k+3} \ldots q^{2k+2n+1}}{(1+q^{2k}) (1+q^{2k+2}) \ldots (1+q^{2k+2n+2})} = \\
\frac{q}{1-q^{2n+2}} \bigg( \frac{-q^{2k+3}\cdot q^{2k+5} \ldots q^{2k+2n+1}}{(1+q^{2k})(1+q^{2k+2})\ldots(1+q^{2k+2n})} +
\frac{q^{2k+3}\cdot q^{2k+5} \ldots q^{2k+2n+1} }{(1+q^{2k+2})(1+q^{2k+4})\ldots(1+q^{2k+2n+2})} \bigg).
\end{multline}
Starting with the left-hand side of \eqref{eq3}, and using the induction hypothesis, we get
\begin{multline} \label{eq4}
\frac{q^{2k+1} \cdot q^{2k+3} \ldots q^{2k+2n+1}}{(1+q^{2k}) (1+q^{2k+2}) \ldots (1+q^{2k+2n+2})} =  \frac{q^{2k+2n+1}}{1+q^{2k+2n+2}} \times \\
\frac{q}{1-q^{2n}} \bigg( \frac{-q^{2k+3}\cdot q^{2k+5} \ldots q^{2k+2n-1}}{(1+q^{2k})(1+q^{2k+2})\ldots(1+q^{2k+2n-2})} +
\frac{q^{2k+3}\cdot q^{2k+5} \ldots q^{2k+2n-1} }{(1+q^{2k+2})(1+q^{2k+4})\ldots(1+q^{2k+2n})} \bigg) .
\end{multline}
Now, one can easily verify that the right-hand side of \eqref{eq4} is equal to the right-hand side of \eqref{eq3}, and this concludes the proof. 
\end{proof}

\begin{lemma} \label{lem2}
For any integer $m \geq 1$, the following holds.
\begin{equation}\label{eq6}
\sum_{k \geq 1} \frac{q^{2k+1} \cdot q^{2k+3} \ldots q^{2k+2m-1}}{(1+q^{2k}) (1+q^{2k+2}) \ldots (1+q^{2k+2m})}=\frac{q^{2m}}{1-q^{2m}} \cdot \frac{q^1 q^3 \ldots q^{2m-1}}{(1+q^2) (1+q^4) \ldots (1+q^{2m})}. 
 \end{equation}
\end{lemma}
\begin{proof}
We will prove the lemma by induction on $m$.
The base case $m=1$ can be easily verified, when we use Identity \eqref{eq1}. Now, supposing that \eqref{eq6} holds for $m \geq 1$, we will show that it also holds for $(m+1)$. Using Lemma \ref{lem1}, rearranging the expression, changing the index and then applying the induction hypothesis, we get what we want:
\begin{align*}
&\sum_{k \geq 1} \frac{q^{2k+1} \cdot q^{2k+3} \ldots q^{2k+2m+1}}{(1+q^{2k}) (1+q^{2k+2}) \ldots (1+q^{2k+2m+2})}=\sum_{k \geq 1} \frac{q}{1-q^{2m+2}} \times \\
&\bigg( \frac{-q^{2k+3} \cdot q^{2k+5} \ldots q^{2k+2m+1}}{(1+q^{2k})(1+q^{2k+2})\ldots(1+q^{2k+2m})} + 
\frac{q^{2k+3}\cdot q^{2k+5} \ldots q^{2k+2m+1} }{(1+q^{2k+2})(1+q^{2k+4})\ldots(1+q^{2k+2m+2})} \bigg)=\frac{q}{1-q^{2m+2}} \times \\
& \bigg( \sum_{k \geq 1} \frac{-q^{2k+3} \cdot q^{2k+5} \ldots q^{2k+2m+1}}{(1+q^{2k})(1+q^{2k+2})\ldots(1+q^{2k+2m})} + 
\sum_{k \geq 1} \frac{q^{2k+3}\cdot q^{2k+5} \ldots q^{2k+2m+1} }{(1+q^{2k+2})(1+q^{2k+4})\ldots(1+q^{2k+2m+2})} \bigg) \\
& =\frac{q}{1-q^{2m+2}} \cdot \bigg( -q^{2m} \cdot \sum_{k \geq 1} \frac{q^{2k+1}\cdot q^{2k+3} \ldots q^{2k+2m-1} }{(1+q^{2k})(1+q^{2k+2})\ldots(1+q^{2k+2m})} \\
&+ \sum_{k \geq 2} \frac{q^{2k+1}\cdot q^{2k+3} \ldots q^{2k+2m-1} }{(1+q^{2k})(1+q^{2k+2})\ldots(1+q^{2k+2m})}\bigg)\\
& =\frac{q}{1-q^{2m+2}} \cdot \bigg( -q^{2m} \cdot 
\frac{q^{2m}}{1-q^{2m}} \cdot \frac{q^1 q^3 \ldots q^{2m-1}}{(1+q^2) (1+q^4) \ldots (1+q^{2m})}  \\
& + \frac{q^{2m}}{1-q^{2m}} \cdot \frac{q^1 q^3 \ldots q^{2m-1}}{(1+q^2) (1+q^4) \ldots (1+q^{2m})} - \frac{q^3 q^5 \ldots q^{2m+1}}{(1+q^2)(1+q^4)\ldots (1+q^{2m+2})}\bigg)\\
&=\frac{q^{2m+2}}{1-q^{2m+2}} \cdot \frac{q^1 q^3 \ldots q^{2m+1}}{(1+q^2)(1+q^4)\ldots(1+q^{2m+2})}. 
\end{align*}
\end{proof}
\begin{lemma} \label{lem3}
For any integers $m_1, m_2 \geq 1$, we have
\begin{align} \label{eq7}
&\sum_{k_2 \geq 1} \sum_{k_1 \geq 1} \frac{q^{2k_1+1}  q^{2k_1+3} \ldots q^{2k_1+2m_1-1} q^{2k_1+2k_2+2m_1+1} q^{2k_1+2k_2+2m_1+3} \ldots q^{2k_1+2k_2+2m_1+2m_2-1}}{(1+q^{2k_1}) (1+q^{2k_1+2}) \ldots (1+q^{2k_1+2m_1})(1+q^{2k_1+2k_2+2m_1})\ldots (1+q^{2k_1+2k_2+2m_1+2m_2})}\\
\nonumber
&=\frac{q^{2(m_1 + m_2)}  q^{2m_2} q^1 q^3 q^5 \ldots q^{2(m_1 + m_2)-1}} {(1-q^{2(m_1 + m_2)})(1-q^{2m_2})(1+q^2)(1+q^4)(1+q^6) \ldots (1+q^{2(m_1+m_2)})}.
\end{align}
\end{lemma}

\begin{proof}
We will prove Identity \eqref{eq7}, by proving the following identity by induction on $m_2$. 
\begin{align} \label{eq8}
&\sum_{k_2 \geq 1} \frac{q^{2k_1+2k_2+2m_1+1} q^{2k_1+2k_2+2m_1+3} \ldots q^{2k_1+2k_2+2m_1+2m_2-1}}{(1+q^{2k_1+2k_2+2m_1})\ldots (1+q^{2k_1+2k_2+2m_1+2m_2})} \\
\nonumber
&=\frac{q^{2m_2} q^{2k_1+2m_1+1} q^{2k_1+2m_1+3} \ldots q^{2k_1+2m_1+2m_2-1}} {(1-q^{2m_2})(1+q^{2k_1+2m_1+2})(1+q^{2k_1+2m_1+4}) \ldots (1+q^{2k_1+2m_1+2m_2})}.
\end{align}
Let us start with the base case $m_2=1$. If we take $k=k_1+k_2+m_1$ in \eqref{eq1}, and compute the sum over $k_2$, we get
\begin{align}
&\sum_{k_2 \geq 1} \frac{q^{2k_1+2k_2+2m_1+1}}{(1+q^{2k_1+2k_2+2m_1})(1+q^{2k_1+2k_2+2m_1+2})}\\
\nonumber
&=\sum_{k_2 \geq 1} \frac{q}{1-q^2} \cdot \bigg( \frac{-1}{1+q^{2k_1+2k_2+2m_1}} + \frac{1}{1+ q^{2k_1+2k_2+2m_1+2}}\bigg)= \frac{q^2}{1-q^2} \frac{q^{2k_1+2m_1+1}}{(1+q^{2k_1+2m_1+2})}.
\end{align} 
Now, suppose that \eqref{eq8} holds for $m_2 \geq 1$. We will prove it for $m_2+1$. Note that, after using Lemma \ref{lem1}, doing some computations and using the induction hypothesis gives what we want.
\begin{align}
&\sum_{k_2 \geq 1} \frac{q^{2k_1+2k_2+2m_1+1} q^{2k_1+2k_2+2m_1+3} \ldots q^{2k_1+2k_2+2m_1+2m_2+1}}{(1+q^{2k_1+2k_2+2m_1})(1+q^{2k_1+2k_2+2m_1+2}) \ldots (1+q^{2k_1+2k_2+2m_1+2m_2+2})} \\
\nonumber
&=\frac{q}{1-q^{2m_2+2}}\bigg(\sum_{k_2 \geq 1} \frac{-q^{2k_1+2k_2+2m_1+3} q^{2k_1+2k_2+2m_1+5} \ldots q^{2k_1+2k_2+2m_1+2m_2+1}}{(1+q^{2k_1+2k_2+2m_1})(1+q^{2k_1+2k_2+2m_1+2}) \ldots (1+q^{2k_1+2k_2+2m_1+2m_2})} \\
\nonumber
&+\sum_{k_2 \geq 1} \frac{q^{2k_1+2k_2+2m_1+3} q^{2k_1+2k_2+2m_1+5} \ldots q^{2k_1+2k_2+2m_1+2m_2+1}}{(1+q^{2k_1+2k_2+2m_1+2})(1+q^{2k_1+2k_2+2m_1+4}) \ldots (1+q^{2k_1+2k_2+2m_1+2m_2+2})}\bigg)\\
\nonumber
&=\frac{q}{1-q^{2m_2+2}} \bigg(-q^{2m_2} \cdot \sum_{k_2 \geq 1} \frac{q^{2k_1+2k_2+2m_1+1} q^{2k_1+2k_2+2m_1+3} \ldots q^{2k_1+2k_2+2m_1+2m_2-1}}{(1+q^{2k_1+2k_2+2m_1})\ldots (1+q^{2k_1+2k_2+2m_1+2m_2})} + \\
\nonumber
&\sum_{k_2 \geq 1} \frac{q^{2k_1+2k_2+2m_1+1} q^{2k_1+2k_2+2m_1+3} \ldots q^{2k_1+2k_2+2m_1+2m_2-1}}{(1+q^{2k_1+2k_2+2m_1})\ldots (1+q^{2k_1+2k_2+2m_1+2m_2})}\\
\nonumber
&- \frac{q^{2k_1+2m_1+3}q^{2k_1+2m_1+5}\ldots q^{2k_1+2m_1+2m_2+1}}{(1+q^{2k_1+2m_1+2})(1+q^{2k_1+2m_1+4})\ldots (1+q^{2k_1+2m_1+2m_2+2})}\bigg)\\
\nonumber
&=\frac{q}{1-q^{2m_2+2}} \bigg(\frac{q^{2m_2} q^{2k_1+2m_1+1} q^{2k_1+2m_1+3} \ldots q^{2k_1+2m_1+2m_2-1}} {(1+q^{2k_1+2m_1+2})(1+q^{2k_1+2m_1+4}) \ldots (1+q^{2k_1+2m_1+2m_2})} \\
\nonumber
&-\frac{q^{2k_1+2m_1+3}q^{2k_1+2m_1+5}\ldots q^{2k_1+2m_1+2m_2+1}}{(1+q^{2k_1+2m_1+2})(1+q^{2k_1+2m_1+4})\ldots (1+q^{2k_1+2m_1+2m_2+2})} \bigg)\\
\nonumber
&=\frac{q^{2m_2+2} q^{2k_1+2m_1+1} q^{2k_1+2m_1+3} \ldots q^{2k_1+2m_1+2m_2+1}} {(1-q^{2m_2+2})(1+q^{2k_1+2m_1+2})(1+q^{2k_1+2m_1+4}) \ldots (1+q^{2k_1+2m_1+2m_2+2})}.
\end{align}
Now, using \eqref{eq8} in the left-hand side of \eqref{eq7}, and then using Lemma \ref{lem2} completes the proof. 
\end{proof}

\begin{lemma} \label{lem4}
For any integer $m_i \geq 1$ where $i=1,2, \ldots, n$, the following identity holds:
\begin{align} \label{biggestsum}
&\sum_{k_n \geq 1} \ldots \sum_{k_2 \geq 1} \sum_{k_1 \geq 1} \bigg(\frac{q^{2k_1+1} q^{2k_1+3} \ldots q^{2k_1+2m_1-1}}{(1+q^{2k_1})(1+q^{2k_1+2})\ldots (1+q^{2k_1+2m_1})}\bigg) \\
\nonumber
&\bigg( \frac{q^{2k_1+2k_2+(2m_1+1)}}{(1+q^{2k_1+2k_2+2m_1})}
\frac{q^{2k_1+2k_2+(2m_1+3)} \ldots q^{2k_1+2k_2+(2m_1+2m_2-1)}}{(1+q^{2k_1+2k_2+2m_1+2})\ldots (1+q^{2k_1+2k_2+2m_1+2m_2})} \bigg) \\
\nonumber
&\bigg( \frac{q^{2k_1+2k_2+2k_3+(2m_1+2m_2+1)}}{(1+q^{2k_1+2k_2+2k_3+2m_1+2m_2})}
\frac{q^{2k_1+2k_2+2k_3+(2m_1+2m_2+3)} \ldots q^{2k_1+2k_2+2k_3+(2m_1+2m_2+2m_3-1)}}{(1+q^{2k_1+2k_2+2k_3+2m_1+2m_2+2})\ldots (1+q^{2k_1+2k_2+2k_3+2m_1+2m_2+2m_3})} \bigg) \\
\nonumber
&\ldots \bigg(\frac{ q^{2k_1+2k_2+\ldots+2k_n+(2m_1+2m_2+\ldots+2m_{n-1}+1)} q^{2k_1+2k_2+\ldots+2k_n+(2m_1+2m_2+\ldots+2m_{n-1}+3)}}{ (1+q^{2k_1+2k_2+\ldots+2k_n+(2m_1+2m_2+\ldots+2m_{n-1})})(1+q^{2k_1+2k_2+\ldots+2k_n+(2m_1+2m_2+\ldots+2m_{n-1}+2)})}\\
\nonumber
&\frac{\ldots q^{2k_1+2k_2+\ldots+2k_n+(2m_1+2m_2+\ldots+2m_{n}-1)}}{\ldots (1+q^{2k_1+2k_2+\ldots+2k_n+(2m_1+2m_2+\ldots+2m_{n-1}+2m_n)})}\bigg)\\
\nonumber
&\textbf{=}\frac{q^1 q^3 \ldots q^{2m_1+2m_2+ \ldots +2m_n -1}}{(1+q^2)(1+q^4)\ldots (1+q^{2m_1+2m_2+ \ldots +2m_n })} \times \\
\nonumber
&\frac{q^{2(m_1+m_2+\ldots+m_n)}}{(1-q^{2(m_1+m_2+\ldots+m_n)})} 
\frac{q^{2(m_2+m_3+\ldots+m_n)}}{(1-q^{2(m_2+m_3+\ldots+m_n)})}
\cdots \frac{q^{2(m_{n-1}+m_n)}}{(1-q^{2(m_{n-1}+m_n)})} \frac{q^{2m_n}}{(1-q^{2m_n})}. 
\end{align}
\end{lemma}

\begin{proof}
This lemma is a generalization of the previous two lemmata, and the proof will be done by induction on $n$. The base case $n=1$ is Lemma \ref{lem2}, while the case $n=2$ is Lemma \ref{lem3}. Let us denote the left-hand side of \eqref{biggestsum} by $B$. Then if we consider \eqref{biggestsum} with $(n+1)$ indices, it is the following expression:

\begin{multline}
B \cdot \sum_{k_{n+1}\geq 1} \bigg(\frac{ q^{2k_1+2k_2+\ldots+2k_{n+1}+(2m_1+2m_2+\ldots+2m_{n}+1)} }{ (1+q^{2k_1+2k_2+\ldots+2k_{n+1}+(2m_1+2m_2+\ldots+2m_{n})})}\times \\
\frac{q^{2k_1+2k_2+\ldots+2k_{n+1}+(2m_1+2m_2+\ldots+2m_{n}+3)} \ldots q^{2k_1+2k_2+\ldots+2k_{n+1}+(2m_1+2m_2+\ldots+2m_{n+1}-1)}}{(1+q^{2k_1+2k_2+\ldots+2k_{n+1}+(2m_1+2m_2+\ldots+2m_{n}+2)})\ldots (1+q^{2k_1+2k_2+\ldots+2k_{n+1}+(2m_1+2m_2+\ldots+2m_{n+1})})}\bigg).
\end{multline}
Now, in \eqref{eq8}, first putting $k_{n+1}$ in place of $k_2$ and $m_{n+1}$ in place of $m_2$, then changing $k_1$ by $(k_1+k_2+\ldots+k_n)$, and changing $m_1$ by $(m_1+m_2+\ldots+m_n)$ we get the sum over $k_{n+1}$ above. Using Identity \eqref{eq8} for the sum over $k_{n+1}$ and then applying the induction hypothesis, we get what we want.
\end{proof}
Now, let us again concentrate on the number of distinct sizes of the slices with shape '$a$' included in the generating function of cylindric partitions with the given profile, and call it $n$. Using the lemmas we proved, we obtain the following:

When $n=0$, we have the generating function $S(q)$. Please recall \eqref{n=0}.

When $n=1$, we will either have a slice with shape $a$ having size $1$ or $(2k+1$) for some $k \geq 1$. So, for size $1$, we have the following generating function, 
$$\frac{q^1}{(1+q^2)}\cdot S(q),$$
while for the odd sizes $\geq3$, we have the following generating function,
$$\frac{q^1}{(1+q^2)}\cdot \frac{q^2}{(1-q^2)} \cdot S(q).$$
Summing these generating functions gives the generating function \eqref{n=1} for case $n=1$:
$$\frac{q^1}{(1+q^2)}\cdot \frac{1}{(1-q^2)}\cdot S(q)=\frac{q^{1^2}}{(q^4;q^4)_1}\cdot S(q).$$
When $n=2$, the possible sizes are 

\noindent $1, ~ 3 $ \\
$1, ~ 2k_1+3 $ for $ ~ k_1 \geq 1$\\
$2k_2+1, ~ 2k_2+3 $ for $ ~ k_2 \geq 1$\\ 
$2k_2+1, ~ 2k_2+2k_1+3 $ for $ ~ k_1, k_2 \geq 1$.

Note that whenever a new index $k_i$ comes into play, we make the distance greater than $2$, that is, while the distance between $1$ and $3$ is $2$, the distance between $1$ and $2k_1+3$ is greater than $2$. Similarly, the distance between $2k_2+1$ and $2k_2+3$ is $2$, but the distance between $2k_2+1$ and $2k_1+2k_2+3$ is greater than $2$. Using the lemmas we proved, one can directly see that we have the following generating functions for each of the following cases separated by the sizes (weights) of slices with shape $a$:

\noindent Sizes $1$ and $3$:
$$ \frac{q^1 q^3}{(1+q^2)(1+q^4)}\cdot S(q).$$ 
Sizes $1$ and $2k_1+3$ for all $k_1\geq1$:
$$\frac{q^1 q^3}{(1+q^2)(1+q^4)} \cdot \frac{q^2}{(1-q^2)} \cdot S(q).$$
Sizes $2k_2+1$ and $2k_2+3$ for all $k_2\geq1$:
$$\frac{q^1 q^3}{(1+q^2)(1+q^4)} \cdot \frac{q^4}{(1-q^4)} \cdot S(q).$$
Sizes $2k_2+1$ and $2k_2+2k_1+3$ for all $k_1, k_2\geq1$:
$$\frac{q^1 q^3}{(1+q^2)(1+q^4)} \cdot \frac{q^2}{(1-q^2)} \cdot \frac{q^4}{(1-q^4)}\cdot S(q).$$
Summing these generating functions, we get the following generating function \eqref{n=2} for the case $n=2$:
$$ \frac{q^1 q^3}{(1+q^2)(1+q^4)} \cdot \frac{1}{(1-q^2)} \cdot \frac{1}{(1-q^4)} \cdot S(q)= \frac{q^{2^2}}{(q^4; q^4)_2} \cdot S(q).$$
When $n=3$, the possible sizes and the corresponding generating functions are as follows:

\noindent Sizes $1$, $3$, $5$:
$$\frac{q^1 q^3 q^5}{(1+q^2)(1+q^4)(1+q^6)} \cdot S(q).$$ 
Sizes $1$, $3$, $2k_1+5$ for all $k_1\geq1$:
$$\frac{q^1 q^3 q^5}{(1+q^2)(1+q^4)(1+q^6)} \cdot \frac{q^2}{(1-q^2)} \cdot S(q).$$
Sizes $1,~ 2k_2+3, ~ 2k_2+5$ for all $k_2\geq1$:
$$\frac{q^1 q^3 q^5}{(1+q^2)(1+q^4)(1+q^6)} \cdot \frac{q^4}{(1-q^4)} \cdot S(q).$$
Sizes $1, ~ 2k_2+3, ~ 2k_2+2k_1+5$ for all $k_1, k_2 \geq 1$:
$$\frac{q^1 q^3 q^5}{(1+q^2)(1+q^4) (1+q^6)} \cdot \frac{q^2}{(1-q^2)} \cdot \frac{q^4}{(1-q^4)} \cdot S(q).$$
Sizes $2k_3+1, ~ 2k_3+3, ~ 2k_3+5$ for all $k_3 \geq 1$:
$$\frac{q^1 q^3 q^5}{(1+q^2)(1+q^4)(1+q^6)} \cdot \frac{q^6}{(1-q^6)} \cdot S(q).$$ 
Sizes $2k_3+1, ~ 2k_3+3, ~ 2k_3+2k_1+5$ for all $k_1, k_3\geq1$:
$$\frac{q^1 q^3 q^5}{(1+q^2)(1+q^4)(1+q^6)} \cdot \frac{q^2}{(1-q^2)} \cdot \frac{q^6}{(1-q^6)} \cdot S(q).$$
Sizes $2k_3+1, ~ 2k_3+2k_2+3,  ~ 2k_3+2k_2+5$ for all $k_2, k_3\geq1$:
$$\frac{q^1 q^3 q^5}{(1+q^2)(1+q^4)(1+q^6)} \cdot \frac{q^4}{(1-q^4)} \cdot \frac{q^6}{(1-q^6)} \cdot S(q).$$
Sizes $2k_3+1, ~ 2k_3+2k_2+3, ~ 2k_3+2k_2+2k_1+5$ for all $k_1, k_2, k_3 \geq 1$:
$$\frac{q^1 q^3 q^5}{(1+q^2)(1+q^4) (1+q^6)} \cdot \frac{q^2}{(1-q^2)} \cdot \frac{q^4}{(1-q^4)} \cdot \frac{q^6}{(1-q^6)} \cdot S(q).$$
So, the generating function \eqref{n=3} for the case $n=3$ :
$$\frac{q^1 q^3 q^5}{(1+q^2)(1+q^4)(1+q^6)} \cdot \frac{1}{(1-q^2)} \cdot \frac{1}{(1-q^4)} \cdot \frac{1}{(1-q^6)} \cdot S(q) = \frac{q^{3^2}}{(q^4; q^4)_3} \cdot S(q).$$
Note that while $n$ changes to $n+1$, the number of possible sizes we deal with changes from $2^n$ to $2^{n+1}$. This depends on two facts: The minimum size is either $1$, or greater than $1$ when we look at the possible sizes, and between two distinct sizes the distance is either $2$, or greater than $2$. 

While $n$ changes to $n+1$, what we do is as follows: 
We take all possible sizes for the case $n$, and add the term $(2n+1)$ to each option, then shift the indices ($k_i$'s) to the right term if they exist. This gives the possible size options including $1$ as a minimum size for the case $n+1$, and these are half of all the possible sizes for the case $n+1$. Then adding $2k_{n+1}$ to the each term of the possible size options including $1$ as a minimum size, we get the other half. 

Application of this process from $n=2$ to $(n+1)=3$ can be seen below:
We take the possible sizes for the case $n=2$, and add the term $5$ to each option, then shift the indices to the right term if they exist:

\noindent
$1, ~ 3 ~ \rightarrow  \mathbf{1, ~ 3, ~ 5} $ \\
$1, ~ 2k_1+3 ~ \rightarrow 1, ~ 2k_1+3,  ~ 5 ~ \rightarrow \mathbf{1, ~ 3,  ~ 2k_1+5} $ \\
$2k_2+1, ~ 2k_2+3 ~ \rightarrow 2k_2+1, ~ 2k_2+3, ~ 5 ~ \rightarrow  \mathbf{1, ~ 2k_2+3, 2k_2+5} $ \\ 
$2k_2+1, ~ 2k_2+2k_1+3 ~ \rightarrow 2k_2+1, ~ 2k_2+2k_1+3, ~ 5 ~ \rightarrow \mathbf{1, ~ 2k_2+3,2k_2+2k_1+5} $.

Now, we add $2k_{3}$ to the each term of the possible size options including $1$ as a minimum size, we get the other half.  

\noindent
$\mathbf{2k_3+1, ~ 2k_3+3, ~ 2k_3+5}$\\
$\mathbf{2k_3+1, ~ 2k_3+3, ~ 2k_3+2k_1+5}$\\
$\mathbf{2k_3+1, ~ 2k_3+2k_2+3,  ~ 2k_3+2k_2+5 }$\\
$\mathbf{2k_3+1, ~ 2k_3+2k_2+3, ~ 2k_3+2k_2+2k_1+5}$

Hence, using these observations and the lemmas we proved, we see that while $n$ changes to $(n+1)$ we multiply the generating function for $n$, by 
$$\frac{q^{2n+1}}{(1+q^{2n+2})} \cdot \frac{1}{(1-q^{2(n+1)})}.$$
So we get
$$\frac{q^{n^2}}{(q^4;q^4)_n} \cdot S(q) \cdot \frac{q^{2n+1}}{(1+q^{2n+2})} \cdot \frac{1}{(1-q^{2(n+1)})}=\frac{q^{(n+1)^2}}{(q^4;q^4)_{n+1}} \cdot S(q),$$
and this completes our proof.
\end{proof}

Now, we continue our work with cylindric partitions having profile $c=(3,0)$, and we find an alternative expression for the generating function of cylindric partitions with the given profile.  

\begin{theorem} \label{theorem2}
Cylindric partitions with profile $c=(3,0)$ have the following generating function:
$$\bigg(\sum_{n\geq 0} \frac{q^{n^2+2 \cdot n}}{(q^4 ; q^4)_n}\bigg) \cdot \frac{(-q^2 ; q^2)_{\infty}}{(q ; q)_{\infty}}. $$
\end{theorem}

\begin{proof}
Cylindric partitions with profile $c = (3, 0)$, have rank $r=2$ and level $\ell=3$, and there are exactly $ \binom{\ell+r-1}{r-1}= $ $\binom{4}{1}$ shapes. The corresponding shapes are $(0), (1), (2), (3)$ and we will denote these shapes by $b, a, c$, and $d$, respectively. Slices that have these shapes with minimum positive weight are shown below. 

\ytableausetup{centertableaux}
\ytableaushort
{\none,\none }
* {4,3}
* [*(gray)]{3,3} ~ ($aq^1$), \ytableausetup{centertableaux}
\ytableaushort
{\none,\none }
* {4,4}
* [*(gray)]{3,3}
~ ($bq^2$),
\ytableausetup{centertableaux}
\ytableaushort
{\none,\none }
* {5,3}
* [*(gray)]{3,3}  
~ ($cq^2$),

\hfill

\ytableausetup{centertableaux}
\ytableaushort
{\none,\none }
* {6,3}
* [*(gray)]{3,3} ~ ($dq^3$).

Since the rank $r$ is $2$, slices with shapes $a$ and $d$ can only have odd weights, while slices with shapes $b$ and $c$ can only have even weights. Slice flow for the given profile can be found below:

\begin{center}
\begin{tikzcd}
                                        &                                         & \color{blue}{dq^3} \arrow[rd]           &                                         & \color{blue}{dq^5} \arrow[rd]           &                                         & \color{blue}{dq^7} \arrow[rd]           &                   \\
                                        & \color{red}{cq^2} \arrow[ru] \arrow[rd] &                                         & \color{red}{cq^4} \arrow[ru] \arrow[rd] &                                         & \color{red}{cq^6} \arrow[rd] \arrow[ru] &                                         & \color{red}{cq^8} \\
\color{red}{aq^1} \arrow[ru] \arrow[rd] &                                         & \color{red}{aq^3} \arrow[ru] \arrow[rd] &                                         & \color{red}{aq^5} \arrow[ru] \arrow[rd] &                                         & \color{red}{aq^7} \arrow[ru] \arrow[rd] &                   \\
                                        & \color{red}{bq^2} \arrow[ru]            &                                         & \color{red}{bq^4} \arrow[ru]            &                                         & \color{red}{bq^6} \arrow[ru]            &                                         & \color{red}{bq^8}

\end{tikzcd}\\
 {\small Table 4.2: Profile $c=(3,0)$}
 \captionlistentry[table]{Profile $c=(3,0)$}
\end{center}
Investigating the slices for the given profile, we see that cylindric partitions with distinct slices including only the slices with shapes $a$, $b$ and $c$ (all shapes except $d$) have the following generating function:
$$\prod_{k \geq 0} (1 +aq^{2k+1}) (1 + bq^{2k+2} + cq^{2k+2}).$$
Whenever we add a slice $dq^{2k+3}$ where $k \geq 0$, we should erase the slices $bq^{2k+2}$ and $bq^{2k+4}$ from the product above. This is the only restriction we have. Let us share with you the possible arrangements of distinct slices with maximum weight $6$ ($bq^6$ or $cq^6$):

$(1+aq^1)(1+bq^2+cq^2)(1+aq^3)(1+bq^4+cq^4)(1+aq^5)(1+bq^6+cq^6)$

$+ (1+aq^1)\mathbf{(1+cq^2)(dq^3)(1+cq^4)}(1+aq^5)(1+bq^6+cq^6)$

$+ (1+aq^1)(1+bq^2+cq^2)(1+aq^3)\mathbf{(1+cq^4)(dq^5)(1+cq^6)}$ 

$+ (1+aq^1)\mathbf{(1+cq^2)(dq^3)(1+cq^4)(dq^5)(1+cq^6)}.$ 

Since the maximum weight is restricted by $6$ above, the possible weights of slices with shape $d$ are $3$ and $5$. So, there is a one-to-one correspondence between the subsets of the set $\{ dq^3, dq^5\}$, and the parts of the sum above. Now, if we compare slice flows for the profile $c=(3,0)$ and the profile $c=(2,1)$ the only difference is that the minimum weight for the shape $d$ is $3$ for the profile $c=(3,0)$, while minimum weight for the shape $a$ is $1$ for the profile $c=(2,1)$. Same pattern, with one missing piece, which is $dq^1$. So, in the proof of Theorem \ref{theorem1}, if we exclude the parts that include the $aq^1$ slice, the generating function becomes
$$\bigg(\sum_{n\geq 0} \frac{q^{(n+1)^2 -1}}{(q^4 ; q^4)_{n}}\bigg) \cdot \frac{(-q^2 ; q^2)_{\infty}}{(q ; q)_{\infty}}=\bigg(\sum_{n\geq 0} \frac{q^{n^2+2 \cdot n}}{(q^4 ; q^4)_n}\bigg) \cdot \frac{(-q^2 ; q^2)_{\infty}}{(q ; q)_{\infty}}. $$
\end{proof}
Now, we will study cylindric partitions with rank $r=3$ and level $\ell=2$. 

\begin{theorem} \label{theorem3}
Cylindric partitions with profile $c=(2,0,0)$ have the following generating function:
$$\bigg(\sum_{n\geq 0} \frac{q^{n^2+2 \cdot n}}{(q^4 ; q^4)_n}\bigg) \cdot \frac{(-q^2 ; q^2)_{\infty}}{(q ; q)_{\infty}}. $$
\end{theorem}

\begin{proof}
Cylindric partitions with profile $c = (2, 0, 0)$, have rank $r=3$ and level $\ell=2$. So there are exactly $ \binom{\ell+r-1}{r-1}= $ $\binom{4}{2}$ shapes. The corresponding shapes are $(0,0), (1,0), (1,1), (2,0), (2,1), (2,2)$ and we will denote these shapes by $c, a, b, d, e$, and $f$, respectively. Slices that have these shapes with minimum positive weight are shown below. Note that the gray squares are created by the profile and have zero weight.

\ytableausetup{centertableaux}
\ytableaushort
{\none,\none }
* {3,2,2}
* [*(gray)]{2,2,2} ~ ($aq^1$), \ytableausetup{centertableaux}
\ytableaushort
{\none,\none }
* {3,3,2}
* [*(gray)]{2,2,2}
~ ($bq^2$),
\ytableausetup{centertableaux}
\ytableaushort
{\none,\none }
* {3,3,3}
* [*(gray)]{2,2,2}  
~ ($cq^3$),

\hfill

\ytableausetup{centertableaux}
\ytableaushort
{\none,\none }
* {4,2,2}
* [*(gray)]{2,2,2} ~ ($dq^2$),
\ytableausetup{centertableaux}
\ytableaushort
{\none,\none }
* {4,3,2}
* [*(gray)]{2,2,2} ~ ($eq^3$),
\ytableausetup{centertableaux}
\ytableaushort
{\none,\none }
* {4,4,2}
* [*(gray)]{2,2,2} ~ ($fq^4$).

Since the rank $r$ is $3$, weights of the slices with shapes $a$ and $f$ are congruent to $1 \pmod{3}$, weights of the slices with shapes $b$ and $d$ are congruent to $2 \pmod{3}$, and weights of the slices with shapes $c$ and $e$ are congruent to $0 \pmod{3}$. Slice flow for the given profile is as follows:
\begin{center}
\begin{tikzcd}
                                        &                                         & \color{blue}{cq^3} \arrow[rd]           &                                         & \color{blue}{dq^5} \arrow[rd]           &                                         & \color{blue}{fq^7} \arrow[rd]           &                   \\
                                        & \color{red}{bq^2} \arrow[ru] \arrow[rd] &                                         & \color{red}{aq^4} \arrow[ru] \arrow[rd] &                                         & \color{red}{eq^6} \arrow[rd] \arrow[ru] &                                         & \color{red}{bq^8} \\
\color{red}{aq^1} \arrow[ru] \arrow[rd] &                                         & \color{red}{eq^3} \arrow[ru] \arrow[rd] &                                         & \color{red}{bq^5} \arrow[ru] \arrow[rd] &                                         & \color{red}{aq^7} \arrow[ru] \arrow[rd] &                   \\
                                        & \color{red}{dq^2} \arrow[ru]            &                                         & \color{red}{fq^4} \arrow[ru]            &                                         & \color{red}{cq^6} \arrow[ru]            &                                         & \color{red}{dq^8}
\end{tikzcd}\\
 {\small Table 4.3: Profile $c=(2,0,0)$}
 \captionlistentry[table]{Profile $c=(2,0,0)$}
\end{center}

Now, consider the following product:
\begin{align*}
&(1 +aq^{1}) \cdot (1 + bq^{2} + dq^{2}) \cdot \underbrace{(1 + eq^{3})}_{bq^2 \leftarrow cq^3 \rightarrow aq^4} \cdot (1 + aq^{4} + fq^{4}) \cdot \underbrace{(1 + bq^{5})}_{aq^4 \leftarrow dq^5 \rightarrow eq^6} \cdot (1 + eq^{6} + cq^{6})\cdot \\
&\underbrace{(1 +aq^{7})}_{eq^6 \leftarrow fq^7 \rightarrow bq^8}  
\cdot (1 + bq^{8} + dq^{8}) \cdot \underbrace{(1 + eq^{9})}_{bq^8 \leftarrow cq^9 \rightarrow aq^{10}} \cdot (1 + aq^{10} + fq^{10}) \cdot \underbrace{(1 + bq^{11})}_{aq^{10} \leftarrow dq^{11} \rightarrow eq^{12}} \ldots.
\end{align*}
In the above product, in every odd weight greater than or equal to $3$, there is a missing slice which should be taken into account, namely $cq^{6k+3}, dq^{6k+5}, fq^{6k+7}$ for any $k \geq 0$. These are indicated under the corresponding parentheses with possible slices next to them. The inserted slice is written in the middle, with its potential neighbors indicated on the left and on the right. Considering these restrictions, possible arrangements of distinct slices with maximum weight $6$ ($eq^6$ or $cq^6$) as follows: 

$(1 +aq^{1}) \cdot (1 + bq^{2} + dq^{2}) \cdot (1 + eq^{3}) \cdot (1 + aq^{4} + fq^{4}) \cdot (1 + bq^{5}) \cdot (1 + eq^{6} + cq^{6})$

$+ (1 +aq^{1}) \cdot \mathbf{(1 + bq^{2}) \cdot (cq^{3}) \cdot (1 + aq^{4} )} \cdot (1 + bq^{5}) \cdot (1 + eq^{6} + cq^{6})$

$+ (1 +aq^{1}) \cdot (1 + bq^{2} + dq^{2}) \cdot (1 + eq^{3}) \cdot \mathbf{(1 + aq^{4}) \cdot (dq^{5}) \cdot (1 + eq^{6})}$

$+ (1 +aq^{1}) \cdot \mathbf{(1 + bq^{2}) \cdot (cq^{3}) \cdot (1 + aq^{4}) \cdot (dq^{5}) \cdot (1 + eq^{6})}.$

This structure and the generalization of this structure are the same as in the proof of Theorem \ref{theorem2}. 
\end{proof}

The following theorem will finalize this section. 

\begin{theorem} \label{theorem4}
Cylindric partitions with profile $c=(1,1,0)$ have the following generating function:
$$\bigg(\sum_{n\geq 0} \frac{q^{n^2}}{(q^4 ; q^4)_n}\bigg) \cdot \frac{(-q^2 ; q^2)_{\infty}}{(q ; q)_{\infty}}. $$
\end{theorem}

\begin{proof}
Cylindric partitions with profile $c = (1, 1, 0)$, have rank $r=3$ and level $\ell=2$. So there are exactly $ \binom{\ell+r-1}{r-1}= $ $\binom{4}{2}$ shapes. The corresponding shapes are $(0,0), (1,0), (1,1), (2,0), (2,1), (2,2)$ and we will denote these shapes by $c, e, a, b, d$, and $f$, respectively. Slices that have these shapes with minimum positive weight are shown below. 

\ytableausetup{centertableaux}
\ytableaushort
{\none,\none }
* {2,2,1}
* [*(gray)]{2,1,1} ~ ($aq^1$), \ytableausetup{centertableaux}
\ytableaushort
{\none,\none }
* {3,1,1}
* [*(gray)]{2,1,1}
~ ($bq^1$),
\ytableausetup{centertableaux}
\ytableaushort
{\none,\none }
* {2,2,2}
* [*(gray)]{2,1,1}  
~ ($cq^2$),
\ytableausetup{centertableaux}
\ytableaushort
{\none,\none }
* {3,2,1}
* [*(gray)]{2,1,1} ~ ($dq^2$),

\hfill

\ytableausetup{centertableaux}
\ytableaushort
{\none,\none }
* {3,2,2}
* [*(gray)]{2,1,1} ~ ($eq^3$),
\ytableausetup{centertableaux}
\ytableaushort
{\none,\none }
* {3,3,1}
* [*(gray)]{2,1,1} ~ ($fq^3$).

Since the rank $r$ is $3$, weights of the slices with shapes $a$ and $b$ are congruent to $1 \pmod{3}$, weights of the slices with shapes $c$ and $d$ are congruent to $2 \pmod{3}$, and weights of the slices with shapes $e$ and $f$ are congruent to $0 \pmod{3}$. Slice flow for the given profile is given below.
\begin{center}
\begin{tikzcd}
\color{blue}{bq^1} \arrow[rd]           &                                         & \color{blue}{fq^3} \arrow[rd]           &                                         & \color{blue}{cq^5} \arrow[rd]           &                                         & \color{blue}{bq^7} \arrow[rd]           &                   \\
                                        & \color{red}{dq^2} \arrow[ru] \arrow[rd] &                                         & \color{red}{aq^4} \arrow[ru] \arrow[rd] &                                         & \color{red}{eq^6} \arrow[rd] \arrow[ru] &                                         & \color{red}{dq^8} \\
\color{red}{aq^1} \arrow[ru] \arrow[rd] &                                         & \color{red}{eq^3} \arrow[ru] \arrow[rd] &                                         & \color{red}{dq^5} \arrow[ru] \arrow[rd] &                                         & \color{red}{aq^7} \arrow[ru] \arrow[rd] &                   \\
                                        & \color{red}{cq^2} \arrow[ru]            &                                         & \color{red}{bq^4} \arrow[ru]            &                                         & \color{red}{fq^6} \arrow[ru]            &                                         & \color{red}{eq^8}
\end{tikzcd}\\
 {\small Table 4.4: Profile $c=(1,1,0)$}
 \captionlistentry[table]{Profile $c=(1,1,0)$}
\end{center}
Now, consider the following product in which the missing slices and the restrictions are indicated:
\begin{align*}
&\underbrace{(1 +aq^{1})}_{bq^1 \rightarrow dq^2} \cdot (1 + cq^{2} + dq^{2}) \cdot \underbrace{(1 + eq^{3})}_{dq^2 \leftarrow fq^3 \rightarrow aq^4} \cdot (1 + aq^{4} + bq^{4}) \cdot \underbrace{(1 + dq^{5})}_{aq^4 \leftarrow cq^5 \rightarrow eq^6} \cdot (1 + eq^{6} + fq^{6})\cdot \\
&\underbrace{(1 +aq^{7})}_{eq^6 \leftarrow bq^7 \rightarrow dq^8} \cdot (1 + cq^{8} + dq^{8}) \cdot \underbrace{(1 + eq^{9})}_{dq^8 \leftarrow fq^9 \rightarrow aq^{10}} \cdot (1 + aq^{10} + bq^{10}) \cdot \underbrace{(1 + dq^{11})}_{aq^{10} \leftarrow cq^{11} \rightarrow eq^{12}} \ldots .
\end{align*}
In the above product, in every odd weight, there is a missing slice which should be taken into account, namely $bq^{6k+1}, fq^{6k+3}, cq^{6k+5}$ for any $k \geq 0$. These are indicated under the corresponding parentheses with possible slices next to them. Possible arrangements of distinct slices with maximum weight $6$ ($eq^6$ or $fq^6$) are as follows:

$(1+aq^1)(1+cq^2+dq^2)(1+eq^3)(1+aq^4+bq^4)(1+dq^5)(1+eq^6+fq^6)$

$+ \mathbf{(bq^1)(1+dq^2)}(1+eq^3)(1+aq^4+bq^4)(1+dq^5)(1+eq^6+fq^6) $ 

$+ (1+aq^1)\mathbf{(1+dq^2)(fq^3)(1+aq^4)}(1+dq^5)(1+eq^6+fq^6) $

$+ (1+aq^1)(1+cq^2+dq^2)(1+eq^3)\mathbf{(1+aq^4)(cq^5)(1+eq^6)} $ 

$+ \mathbf{(bq^1)(1+dq^2)(fq^3)(1+aq^4)}(1+dq^5)(1+eq^6+fq^6) $

$+ \mathbf{(bq^1)(1+dq^2)}(1+eq^3)\mathbf{(1+aq^4)(cq^5)(1+eq^6)} $

$+ (1+aq^1)\mathbf{(1+dq^2)(fq^3)(1+aq^4)(cq^5)(1+eq^6)} $ 

$+ \mathbf{(bq^1)(1+dq^2)(fq^3)(1+aq^4)(cq^5)(1+eq^6)}$. 

This structure and generalization of this structure are the same as in the proof of Theorem \ref{theorem1}, then so is the proof. 
\end{proof}

\section{Method of Weighted Words on Cylindric Partitions with Rank 2 - Level 4 and Cylindric Partitions with Rank 4 - Level 2}
\label{sec:3}
In this section, we apply the method of weighted words on cylindric partitions with rank $2$ - level $4$ and rank $4$ - level $2$. First of all, we begin studying cylindric partitions with rank $2$ - level $4$. In this case, we deal with profiles $c=(4,0)$, $c=(3,1)$ and $c=(2,2)$. We derive alternative expressions for the generating functions of cylindric partitions with the given profiles.  

\begin{theorem} \label{theorem7}
Cylindric partitions with profile $c=(4,0)$ have the following generating function:
$$\bigg(\sum_{n\geq 0} \frac{q^{n(n+1)} \cdot (-q^2 ; q^2)_n }{(-q^3 ; q^2)_n \cdot (q^2 ; q^2)_n}\bigg) \cdot \frac{(-q^3 ; q^2)_{\infty}}{(q ; q)_{\infty}}. $$
\end{theorem}

\begin{proof}
Slices of cylindric partitions with profile $c = (4, 0)$ can have exactly $ \binom{\ell+r-1}{r-1}= $ $\binom{5}{1}$ shapes. The possible shapes are $(0), (1), (2), (3), (4)$ and we will denote these shapes by $b, a, c, d$ and $e$, respectively. Slices that have these shapes with minimum positive weight are shown below. 

\ytableausetup{centertableaux}
\ytableaushort
{\none,\none }
* {5,4}
* [*(gray)]{4,4} ~ ($aq^1$), \ytableausetup{centertableaux}
\ytableaushort
{\none,\none }
* {5,5}
* [*(gray)]{4,4}
~ ($bq^2$),
\ytableausetup{centertableaux}
\ytableaushort
{\none,\none }
* {6,4}
* [*(gray)]{4,4}  
~ ($cq^2$),

\hfill

\ytableausetup{centertableaux}
\ytableaushort
{\none,\none }
* {7,4}
* [*(gray)]{4,4} ~ ($dq^3$),
\ytableausetup{centertableaux}
\ytableaushort
{\none,\none }
* {8,4}
* [*(gray)]{4,4} ~ ($eq^4$).

Slices with shapes $a$ and $d$ can only have odd weights, while slices with shapes $b, c$ and $e$ can only have even weights. Slice flow for the given profile can be found below.
\begin{center}
\begin{tikzcd}
                                        & \color{blue}{bq^2} \arrow[rd]           &                                         & \color{blue}{bq^4} \arrow[rd]           &                                         & \color{blue}{bq^6} \arrow[rd]           &                                         & \color{blue}{bq^8} \\
\color{red}{aq^1} \arrow[ru] \arrow[rd] &                                         & \color{red}{aq^3} \arrow[ru] \arrow[rd] &                                         & \color{red}{aq^5} \arrow[ru] \arrow[rd] &                                         & \color{red}{aq^7} \arrow[ru] \arrow[rd] &                    \\
                                        & \color{red}{cq^2} \arrow[rd] \arrow[ru] &                                         & \color{red}{cq^4} \arrow[rd] \arrow[ru] &                                         & \color{red}{cq^6} \arrow[rd] \arrow[ru] &                                         & \color{red}{cq^8}  \\
                                        &                                         & \color{red}{dq^3} \arrow[ru] \arrow[rd] &                                         & \color{red}{dq^5} \arrow[ru] \arrow[rd] &                                         & \color{red}{dq^7} \arrow[ru] \arrow[rd] &                    \\
                                        &                                         &                                         & \color{blue}{eq^4} \arrow[ru]           &                                         & \color{blue}{eq^6} \arrow[ru]           &                                         & \color{blue}{eq^8}
\end{tikzcd}\\
 {\small Table 5.1: Profile $c=(4,0)$}
 \captionlistentry[table]{Profile $c=(4,0)$}
\end{center}
Now, we will work on the following product in which the missing slices and the restrictions are indicated:
\begin{align*}
(1+aq^1) \cdot &\underbrace{(1 +cq^2)}_{bq^2 \rightarrow aq^3} \cdot (1 + aq^3 + dq^3) \cdot\underbrace{(1+cq^4)}_{\underbrace{aq^3 \leftarrow bq^4 \rightarrow aq^5}_{ dq^3 \leftarrow eq^4 \rightarrow dq^5} } \cdot (1 + aq^5 + dq^5) \cdot \underbrace{(1+cq^6)\ldots}_{ \underbrace{aq^5 \leftarrow bq^6 \rightarrow aq^7}_{dq^5 \leftarrow eq^6 \rightarrow dq^7}} 
\end{align*}
The notation is explained below. Now, let us define the following product which is the generating function of cylindric partitions created by slices with shapes $a, c$ or $d$, that is, shapes $b$ and $e$ are not included, when repetitions of slices are not allowed:
\begin{equation} \label{P(q) c=(4,0)}
P(q):= (1 + aq) \prod_{k \geq 0} (1 +cq^{2k+2})(1+aq^{2k+3}+dq^{2k+3}).
\end{equation}
Note that in the product above, for the weight $2$ there is only one missing slice, namely $bq^2$ which can be followed by with slice $aq^3$. Except this case, for every weight $(2k+4)$, there are two slices missing: $bq^{2k+4}$ and $eq^{2k+4}$ for $k \geq 0$. If we put $bq^{2k+4}$ in place of $(1+cq^{2k+4})$ then we need to erase the slices $dq^{2k+3}$ and $dq^{2k+5}$, while if we put $eq^{2k+4}$ in place of $(1+cq^{2k+4})$ then we need to erase the slices $aq^{2k+3}$ and $aq^{2k+5}$. A small example of this process is given below, in which we restrict the maximum slice weight with $7$, and we give the beginning terms of the generating function of cylindric partitions with the given profile, when repetitions of slices are \textbf{not} allowed.     

$(1 + aq^1)(1 +cq^{2})(1+aq^{3}+dq^{3})(1 +cq^{4})(1+aq^{5}+dq^{5})(1 +cq^{6})(1+aq^{7}+dq^{7})$

$+ (1 + aq^1)\mathbf{(bq^{2})(1+aq^{3})}(1 +cq^{4})(1+aq^{5}+dq^{5})(1 +cq^{6})(1+aq^{7}+dq^{7})$

$+ (1 + aq^1)(1 +cq^{2})\mathbf{(1+aq^{3})(bq^{4})(1+aq^{5})}(1 +cq^{6})(1+aq^{7}+dq^{7})$

$+ (1 + aq^1)(1 +cq^{2})\mathbf{(1+dq^{3})(eq^{4})(1+dq^{5})}(1 +cq^{6})(1+aq^{7}+dq^{7})$

$+(1 + aq^1)(1 +cq^{2})(1+aq^{3}+dq^{3})(1 +cq^{4})\mathbf{(1+aq^{5})(bq^{6})(1+aq^{7})}$

$+ (1 + aq^1)(1 +cq^{2})(1+aq^{3}+dq^{3})(1 +cq^{4})\mathbf{(1+dq^{5})(eq^{6})(1+dq^{7})}$

$+ (1 + aq^1)\mathbf{(bq^{2})(1+aq^{3})}\mathbf{(bq^{4})(1+aq^{5})}(1 +cq^{6})(1+aq^{7}+dq^{7})$

$+(1 + aq^1)\mathbf{(bq^{2})(1+aq^{3})}(1 +cq^{4})\mathbf{(1+aq^{5})(bq^{6})(1+aq^{7})}$

$+(1 + aq^1)\mathbf{(bq^{2})(1+aq^{3})}(1 +cq^{4})\mathbf{(1+dq^{5})(eq^{6})(1+dq^{7})}$

$+(1 + aq^1)(1 +cq^{2})\mathbf{(1+aq^{3})(bq^{4})(1+aq^{5})}\mathbf{(bq^{6})(1+aq^{7})}$

$+ (1 + aq^1)(1 +cq^{2})\mathbf{(1+dq^{3})(eq^{4})(1+dq^{5})}\mathbf{(eq^{6})(1+dq^{7})}$

$+ (1 + aq^1)\mathbf{(bq^{2})(1+aq^{3})(bq^{4})(1+aq^{5})}\mathbf{(bq^{6})(1+aq^{7})}.$

Let us write the version of the sum above after allowing slices to repeat and replacing $a,b,c,d,e$ by $1$:

$ (\frac{1}{1-q}) (\frac{1}{1-q^2}) (\frac{1+q^3}{1-q^3}) (\frac{1}{1-q^4})(\frac{1+q^5}{1-q^5})(\frac{1}{1-q^6})(\frac{1+q^7}{1-q^7})$

$+ (\frac{1}{1-q}) (\frac{q^2}{1-q^2}) (\frac{1}{1-q^3}) (\frac{1}{1-q^4})(\frac{1+q^5}{1-q^5})(\frac{1}{1-q^6})(\frac{1+q^7}{1-q^7})$

$+(\frac{1}{1-q}) (\frac{1}{1-q^2}) (\frac{1}{1-q^3}) (\frac{q^4}{1-q^4})(\frac{1}{1-q^5})(\frac{1}{1-q^6})(\frac{1+q^7}{1-q^7})$

$+ (\frac{1}{1-q}) (\frac{1}{1-q^2}) (\frac{1}{1-q^3}) (\frac{q^4}{1-q^4})(\frac{1}{1-q^5})(\frac{1}{1-q^6})(\frac{1+q^7}{1-q^7})$

$+ (\frac{1}{1-q}) (\frac{1}{1-q^2}) (\frac{1+q^3}{1-q^3}) (\frac{1}{1-q^4})(\frac{1}{1-q^5})(\frac{q^6}{1-q^6})(\frac{1}{1-q^7})$

$+ (\frac{1}{1-q}) (\frac{1}{1-q^2}) (\frac{1+q^3}{1-q^3}) (\frac{1}{1-q^4})(\frac{1}{1-q^5})(\frac{q^6}{1-q^6})(\frac{1}{1-q^7})$

$+ (\frac{1}{1-q}) (\frac{q^2}{1-q^2}) (\frac{1}{1-q^3}) (\frac{q^4}{1-q^4})(\frac{1}{1-q^5})(\frac{1}{1-q^6})(\frac{1+q^7}{1-q^7})$

$+ (\frac{1}{1-q}) (\frac{q^2}{1-q^2}) (\frac{1}{1-q^3}) (\frac{1}{1-q^4})(\frac{1}{1-q^5})(\frac{q^6}{1-q^6})(\frac{1}{1-q^7})$

$+ (\frac{1}{1-q}) (\frac{q^2}{1-q^2}) (\frac{1}{1-q^3}) (\frac{1}{1-q^4})(\frac{1}{1-q^5})(\frac{q^6}{1-q^6})(\frac{1}{1-q^7})$

$+ (\frac{1}{1-q}) (\frac{1}{1-q^2}) (\frac{1}{1-q^3}) (\frac{q^4}{1-q^4})(\frac{1}{1-q^5})(\frac{q^6}{1-q^6})(\frac{1}{1-q^7})$

$+ (\frac{1}{1-q}) (\frac{1}{1-q^2}) (\frac{1}{1-q^3}) (\frac{q^4}{1-q^4})(\frac{1}{1-q^5})(\frac{q^6}{1-q^6})(\frac{1}{1-q^7})$

$+(\frac{1}{1-q}) (\frac{q^2}{1-q^2}) (\frac{1}{1-q^3}) (\frac{q^4}{1-q^4})(\frac{1}{1-q^5})(\frac{q^6}{1-q^6})(\frac{1}{1-q^7}).$

Now, we will generalize this small case. First of all, allowing slices to repeat and then putting $a=c=d=1$, \eqref{P(q) c=(4,0)} becomes
\begin{equation} \label{S(q) c=(4,0)}
S(q):= \frac{\prod_{k\geq0} (1+q^{2k+3})}{\prod_{k\geq0}(1-q^{k+1})}=\frac{(-q^3;q^2)_{\infty}}{(q;q)_{\infty}}.
\end{equation}
Let $n$ denote the number of distinct slices with shape $b$ or $e$ that appear in the cylindric partitions with the given profile. Note that, since we allow slices to repeat, for instance if $bq^2$ comes into the picture, then we will see it as $\frac{bq^2}{1-bq^2}$, i.e., all repeated forms of $bq^2$ are being taken into account.

When $n=0$, the generating function of cylindric partitions with the given profile is \eqref{S(q) c=(4,0)}.

When $n=1$, we will either have a slice with shape $b$ ($bq^{2k+2}$ for some $k\geq 0$), or a slice with shape $e$ ($eq^{2k+4}$ for some $k\geq 0$). So, the generating function for $n=1$ is

$$ \underbrace{\bigg(\frac{q^2}{1+q^3} + \sum_{k\geq1} \frac{q^{2k+2}}{(1+q^{2k+1})(1+q^{2k+3})}\bigg)\cdot S(q)}_{\text{every possible size of a slice with shape b included}} + \underbrace{\bigg(\sum_{k\geq1} \frac{q^{2k+2}}{(1+q^{2k+1})(1+q^{2k+3})}\bigg)\cdot S(q)}_{\text{every possible size of a slice with shape e included}}.$$

When $n=2$, we will have two shapes of $b$ or $e$ so that one of the following cases may happen: $(b,b)$, $(b,e)$, $(e,e)$ or $(e,b)$. Note that, in these duos the former written shapes has less weight than the latter. 

Case 1: $(b,b)$

\noindent If weights of slices $(b,b)$ are exactly $2$ and $4$, the generating function including this duo is 
$$\bigg(\frac{q^2 \cdot q^4}{(1+q^3) (1+q^5)}\bigg)\cdot S(q),$$
if weights of slices $(b,b)$ are $2$ and $\geq6$, the generating function including this duo is
$$\bigg(\sum_{k\geq 1} \frac{q^2 \cdot q^{2k+4}}{(1+q^3)(1+q^{2k+3})(1+q^{2k+5})}\bigg)\cdot S(q),$$
if weights of slices $(b,b)$ are $2k+2$ and $2k+4$ for $k\geq1$, the generating function including this duo is
$$\bigg( \sum_{k\geq 1} \frac{q^{2k+2} \cdot q^{2k+4}}{(1+q^{2k+1})(1+q^{2k+3})(1+q^{2k+5})}\bigg)\cdot S(q),$$
if weights of slices $(b,b)$ are $2k+2$ and $\geq2k+6$ for $k\geq1$, the generating function including this duo is
$$\bigg( \sum_{m\geq1} \sum_{k\geq 1} \frac{q^{2k+2} \cdot q^{2m+2k+4}}{(1+q^{2k+1})(1+q^{2k+3})(1+q^{2m+2k+3})(1+q^{2m+2k+5})}\bigg)\cdot S(q).$$
The generating function for this case is the sum of these four disjoint cases.

Case 2: $(b, e)$

Since the minimum distance between the weights of slices with shape $b$ and $e$ is $4$, the generating function corresponding to this case:
\begin{align*}
&\bigg(\sum_{k\geq 1} \frac{q^2 \cdot q^{2k+4}}{(1+q^3) (1+q^{2k+3})(1+q^{2k+5})}\bigg)\cdot S(q) + \\
&\bigg( \sum_{m\geq1} \sum_{k\geq 1} \frac{q^{2k+2} \cdot q^{2m+2k+4}}{(1+q^{2k+1})(1+q^{2k+3})(1+q^{2m+2k+3})(1+q^{2m+2k+5})}\bigg)\cdot S(q).
\end{align*}

Case 3: $(e, e)$

As the minimum weight of a slice with shape $e$ is $4$, we have the following generating function for this case:
\begin{align*}
&\bigg( \sum_{k\geq 1} \frac{q^{2k+2} \cdot q^{2k+4}}{(1+q^{2k+1})(1+q^{2k+3})(1+q^{2k+5})}\bigg)\cdot S(q) + \\
&\bigg( \sum_{m\geq1} \sum_{k\geq 1} \frac{q^{2k+2} \cdot q^{2m+2k+4}}{(1+q^{2k+1})(1+q^{2k+3})(1+q^{2m+2k+3})(1+q^{2m+2k+5})}\bigg)\cdot S(q).
\end{align*}

Case 4: $(e, b)$

Since the minimum distance between the weights of slices with shape $e$ and $b$ is $4$ and the minimum weight of a slice with shape $e$ is $4$, the generating function for this case is:
$$\bigg( \sum_{m\geq1} \sum_{k\geq 1} \frac{q^{2k+2} \cdot q^{2m+2k+4}}{(1+q^{2k+1})(1+q^{2k+3})(1+q^{2m+2k+3})(1+q^{2m+2k+5})}\bigg)\cdot S(q).$$

Now, we will share some lemmas which will help us in the computation of the sums above. Since the proofs of these lemmas are almost the same as in the proof of Theorem \ref{theorem1}, they are omitted.

\begin{lemma}\ \label{lemma5}
For any integer $k \geq 0$, we have the following identity:
\begin{equation*} 
\frac{q^{2k+2}}{(1+q^{2k+1}) (1+q^{2k+3})}=\frac{q}{(1-q^2)} \cdot \bigg( \frac{-1}{1+q^{2k+1}} + \frac{1}{1+ q^{2k+3}} \bigg).    
\end{equation*}
For any integers $k \geq 0$, $n \geq 2$ we have the following identity,
\begin{multline*} 
\frac{q^{2k+2} \cdot q^{2k+4} \ldots q^{2k+2n}}{(1+q^{2k+1}) (1+q^{2k+3}) \ldots (1+q^{2k+2n+1})} \\
= \frac{q}{(1-q^{2n})} \bigg( \frac{-q^{2k+4}\cdot q^{2k+6} \ldots q^{2k+2n}}{(1+q^{2k+1})(1+q^{2k+3})\ldots(1+q^{2k+2n-1})} +
\frac{q^{2k+4}\cdot q^{2k+6} \ldots q^{2k+2n} }{(1+q^{2k+3})(1+q^{2k+5})\ldots(1+q^{2k+2n+1})} \bigg).
\end{multline*}
\end{lemma}

\begin{lemma} \label{lemma6}
For any integer $m \geq 1$, the following holds.
\begin{equation*}
\sum_{k \geq 1} \frac{q^{2k+2} \cdot q^{2k+4} \ldots q^{2k+2m}}{(1+q^{2k+1}) (1+q^{2k+3}) \ldots (1+q^{2k+2m+1})}=\frac{q^{2m}}{(1-q^{2m})} \cdot \frac{q^2 q^4 \ldots q^{2m}}{(1+q^3) (1+q^5) \ldots (1+q^{2m+1})}. 
 \end{equation*}
\end{lemma}

\begin{lemma} \label{lemma7}
For any integers $m_1, m_2 \geq 1$, we have
\begin{align*}
&\sum_{k_2 \geq 1} \sum_{k_1 \geq 1} \frac{q^{2k_1+2}  q^{2k_1+4} \ldots q^{2k_1+2m_1} q^{2k_1+2k_2+2m_1+2} q^{2k_1+2k_2+2m_1+4} \ldots }{(1+q^{2k_1+1}) (1+q^{2k_1+3}) \ldots (1+q^{2k_1+2m_1+1})(1+q^{2k_1+2k_2+2m_1+1})\ldots }\\
&\frac{q^{2k_1+2k_2+2m_1+2m_2-2} q^{2k_1+2k_2+2m_1+2m_2}}{(1+q^{2k_1+2k_2+2m_1+2m_2-1})(1+q^{2k_1+2k_2+2m_1+2m_2+1})}\\
&=\frac{q^{2(m_1 + m_2)}  q^{2m_2}}{(1-q^{2(m_1 + m_2)})(1-q^{2m_2})}\cdot \frac{ q^2 q^4 q^6 \ldots q^{2(m_1 + m_2)}} {(1+q^3)(1+q^5)(1+q^7) \ldots (1+q^{2(m_1+m_2)+1})}.
\end{align*}
\end{lemma}

\begin{lemma} \label{lemma8}
For any integer $m_i \geq 1$ where $i=1,2, \ldots, n$, the following equality holds:
\begin{align*} 
&\sum_{k_n \geq 1} \ldots \sum_{k_2 \geq 1} \sum_{k_1 \geq 1} \bigg(\frac{q^{2k_1+2} q^{2k_1+4} \ldots q^{2k_1+2m_1}}{(1+q^{2k_1+1})(1+q^{2k_1+3})\ldots (1+q^{2k_1+2m_1+1})}\bigg) \\
&\bigg( \frac{q^{2k_1+2k_2+(2m_1+2)}}{(1+q^{2k_1+2k_2+2m_1+1})}\frac{q^{2k_1+2k_2+(2m_1+4)} \ldots q^{2k_1+2k_2+(2m_1+2m_2)}}{(1+q^{2k_1+2k_2+2m_1+3})\ldots (1+q^{2k_1+2k_2+2m_1+2m_2+1})} \bigg) \\
&\bigg( \frac{q^{2k_1+2k_2+2k_3+(2m_1+2m_2+2)}}{(1+q^{2k_1+2k_2+2k_3+2m_1+2m_2+1})} \frac{q^{2k_1+2k_2+2k_3+(2m_1+2m_2+4)} \ldots }{(1+q^{2k_1+2k_2+2k_3+2m_1+2m_2+3})\ldots }  \ldots \\
&\frac{q^{2k_1+2k_2+2k_3+(2m_1+2m_2+2m_3)}}{(1+q^{2k_1+2k_2+2k_3+2m_1+2m_2+2m_3+1})} \bigg) \ldots \bigg( \frac{q^{2k_1+2k_2+\ldots+2k_n+(2m_1+2m_2+\ldots+2m_{n-1}+2)}}{(1+q^{2k_1+2k_2+\ldots+2k_n+(2m_1+2m_2+\ldots+2m_{n-1}+1)})}
\\
&\frac{q^{2k_1+2k_2+\ldots+2k_n+(2m_1+2m_2+\ldots+2m_{n-1}+4)} \ldots q^{2k_1+2k_2+\ldots+2k_n+(2m_1+2m_2+\ldots+2m_{n})}}{(1+q^{2k_1+2k_2+\ldots+2k_n+(2m_1+2m_2+\ldots+2m_{n-1}+3)})\ldots (1+q^{2k_1+2k_2+\ldots+2k_n+(2m_1+2m_2+\ldots+2m_{n-1}+2m_n+1)})}\bigg)\\
&\textbf{=}\frac{q^2 q^4 \ldots q^{2m_1+2m_2+ \ldots +2m_n}}{(1+q^3)(1+q^5)\ldots (1+q^{2m_1+2m_2+ \ldots +2m_n+1})} \times \\
&\frac{q^{2(m_1+m_2+\ldots+m_n)}}{1-q^{2(m_1+m_2+\ldots+m_n)}} 
\frac{q^{2(m_2+m_3+\ldots+m_n)}}{1-q^{2(m_2+m_3+\ldots+m_n)}}
\cdots \frac{q^{2(m_{n-1}+m_n)}}{1-q^{2(m_{n-1}+m_n)}} \frac{q^{2m_n}}{1-q^{2m_n}} 
\end{align*}
\end{lemma}

Now, using Lemma \ref{lemma8}, let us compute the sums for the cases $n=1$ and $n=2$. 

If $\underline{n=1}$:
$$ \underbrace{\bigg(\frac{q^2}{(1+q^3)} + \sum_{k\geq1} \frac{q^{2k+2}}{(1+q^{2k+1})(1+q^{2k+3})}\bigg)\cdot S(q)}_{\text{every possible size of a slice with shape b included}}= \frac{q^2}{(1+q^3)(1-q^2)}\cdot S(q)$$ 
$$ \underbrace{\bigg(\sum_{k\geq1} \frac{q^{2k+2}}{(1+q^{2k+1})(1+q^{2k+3})}\bigg)\cdot S(q)}_{\text{every possible size of a slice with shape e included}}=\frac{q^4}{(1+q^3)(1-q^2)}\cdot S(q)$$
Note that, because of Lemma \ref{lemma8}, there are two things that are not a coincidence: 
having the same denominators for two disjoint cases above and having the minimal weights coming from each disjoint case as a power of $q$. Recall that the minimal positive weight for a slice with shape $b$ is $2$, while the minimal weight for a slice with shape $e$ is $4$.   
So, the generating function for $n=1$ is:
$$ \frac{(q^2 + q^4)}{(1+q^3)(1-q^2)}\cdot S(q) = \frac{q^2 (1+q^2)}{(1+q^3)(1-q^2)}\cdot S(q)= \frac{q^{1\cdot (1+1)} \cdot (-q^2;q^2)_1}{(-q^3;q^2)_1 (q^2;q^2)_1}\cdot\frac{(-q^3;q^2)_{\infty}}{(q;q)_{\infty}}.$$
If $\underline{n=2}$, we have the following generating functions corresponding to each case:

\noindent Case 1: $(b,b)$
$$\bigg(\frac{q^2 \cdot q^4}{(1+q^3)(1+q^5)(1-q^2)(1-q^4)}\bigg)\cdot S(q),$$
Case 2: $(b, e)$
$$\bigg(\frac{q^2 \cdot q^6}{(1+q^3)(1+q^5)(1-q^2)(1-q^4)}\bigg)\cdot S(q),$$
Case 3: $(e, e)$
$$\bigg(\frac{q^4 \cdot q^6}{(1+q^3)(1+q^5)(1-q^2)(1-q^4)}\bigg)\cdot S(q),$$
Case 4: $(e, b)$
$$\bigg(\frac{q^4 \cdot q^8}{(1+q^3)(1+q^5)(1-q^2)(1-q^4)}\bigg)\cdot S(q).$$
So, the generating function for $n=2$ is
$$\bigg(\frac{q^2 \cdot q^4 \cdot (1+q^2)(1+q^4)}{(1+q^3)(1+q^5)(1-q^2)(1-q^4)}\bigg)\cdot S(q)= \frac{q^{2\cdot (2+1)} \cdot (-q^2;q^2)_2}{(-q^3;q^2)_2 (q^2;q^2)_2}\cdot\frac{(-q^3;q^2)_{\infty}}{(q;q)_{\infty}}.$$
Now, we will prove the following generating function for any $n$:
$$\bigg( \frac{q^{n\cdot(n+1)} \cdot (-q^2 ; q^2)_n }{(-q^3 ; q^2)_n \cdot (q^2 ; q^2)_n}\bigg) \cdot \frac{(-q^3 ; q^2)_{\infty}}{(q ; q)_{\infty}}. $$
Note that cases $n=1$ and $n=2$ were already proven above. It is easy to see that the base case $n=0$ also holds. We will proceed by induction on $n$. Now, supposing the above formula holds for $n$, we will prove it for $(n+1)$. 

Let us give you a summary of what happens when $n$ changes to $(n+1)$: First of all, the number of indices ($k_i$'s) that we see increases by one (changes from $n$ to $n+1$). So, by Lemma \ref{lemma8}, we multiply the denominator by $(1+q^{2(n+1)+1})(1-q^{2(n+1)})$. Recall that, in the numerator, we see the minimal weights coming from disjoint cases as a power of $q$. While $n$ changes to $(n+1)$ the number of disjoint cases we deal with changes from $2^n$ to $2^{n+1}$.

For a general $n$, the minimum weight comes from the case $\underbrace{(b, b, b, \ldots, b)}_{\text{$n$ many}}$ which is 
$$ 2 + 4 + 6 + \ldots + 2 \cdot n= n \cdot (n+1). $$
Starting with this case, how we compute the minimal weights for other cases is as follows: Let us work on $n=3$, for the sake of simplicity. Starting with $(b, b, b)$, the minimal triple is $(bq^2, bq^4, bq^6)$. If we change the biggest slice with shape $b$ with a slice with shape $e$ (as the minimum distance between weights of slices with shapes $b$ and $e$ is for $4$), it becomes $(bq^2, bq^4, eq^8)$, and we get  
$$\underbrace{(bq^2, bq^4,bq^6)}_{(q^2 \cdot q^4 \cdot q^6)} \rightarrow \underbrace{(bq^2, bq^4, bq^6), (bq^2, bq^4, eq^8)}_{(q^2 \cdot q^4 \cdot q^6) (1 + q^2)}.$$
Continuing with the cases $(bq^2, bq^4, bq^6), (bq^2, bq^4, eq^8)$, we will change the roles of $b$ and $e$ for the two biggest slices and we will get two more cases, that is, we will get $(bq^2, eq^6, eq^8), (bq^2, eq^6, bq^{10})$:
$$ \underbrace{(bq^2, bq^4, bq^6), (bq^2, bq^4, eq^8)}_{(q^2 \cdot q^4 \cdot q^6) (1 + q^2)} \rightarrow \underbrace{(bq^2, bq^4, bq^6), (bq^2, bq^4, eq^8),(bq^2, eq^6, eq^8), (bq^2, eq^6, bq^{10})}_{(q^2 \cdot q^4 \cdot q^6) (1 + q^2)(1+q^4)}.$$
Finally, starting with $(bq^2, bq^4, bq^6), (bq^2, bq^4, eq^8),(bq^2, eq^6, eq^8), (bq^2, eq^6, bq^{10})$, and interchanging the roles of $b$ and $e$ for all the slices, we will get the last cases: $(eq^4, eq^6, eq^8), (eq^4, eq^6, bq^{10}),(eq^4, bq^8, bq^{10}),$ $(eq^4, bq^8, eq^{12})$, so we obtain all the cases:

\noindent $(bq^2, bq^4, bq^6), (bq^2, bq^4, eq^8),(bq^2, eq^6, eq^8), (bq^2, eq^6, bq^{10}), \mathbf{(eq^4, eq^6, eq^8), (eq^4, eq^6, bq^{10})},$ \\
$ \mathbf{(eq^4, bq^8, bq^{10}), (eq^4, bq^8, eq^{12})}$.

One can easily see that the weights of the last four cases are the iterations of the weights of the first four cases, respectively. So, we finally get:  
$$(q^2 \cdot q^4 \cdot q^6) (1 + q^2)(1+q^4)(1+q^6).$$
We apply the same process for all $n$, to get all the minimal weights. So, when $n$ changes to $(n+1)$, the number of change in the roles of shapes $b$ and $e$ changes from $n$ to $n+1$. As a result, we multiply the numerator by $q^{2(n+1)} \cdot (1 + q^{2(n+1)})$.
Therefore, we obtain:
$$\bigg( \frac{q^{n\cdot(n+1)} \cdot (-q^2 ; q^2)_n \cdot q^{2(n+1)} \cdot (1 + q^{2(n+1)}) }{(-q^3 ; q^2)_n \cdot (q^2 ; q^2)_n \cdot (1+q^{2(n+1)+1})(1-q^{2.(n+1)})} \bigg) \cdot \frac{(-q^3 ; q^2)_{\infty}}{(q ; q)_{\infty}}. $$
This is nothing but
$$\bigg(\frac{q^{(n+1)(n+2)} \cdot (-q^2 ; q^2)_{n+1} }{(-q^3: q^2)_{n+1} \cdot (q^2 ; q^2)_{n+1}}\bigg) \cdot \frac{(-q^3 ; q^2)_{\infty}}{(q ; q)_{\infty}}.$$
\end{proof}

\begin{theorem} \label{theorem8}
Cylindric partitions with profile $c=(2,2)$ have the following generating function:
$$\bigg(1 + \sum_{n\geq 1} \frac{2\cdot q^{n(n+1)} \cdot (-q^2 ; q^2)_{n-1} }{(q^2 ; q^2)_n \cdot (-q ; q^2)_n}\bigg) \cdot \frac{(-q ; q^2)_{\infty}}{(q ; q)_{\infty}}. $$
\end{theorem}

\begin{proof}
Slices of cylindric partitions with profile $c = (2, 2)$ can have exactly $ \binom{\ell+r-1}{r-1}= $ $\binom{5}{1}$ shapes. The possible shapes are $(0), (1), (2), (3), (4)$ and we will denote these shapes by $d, a, c, b$ and $e$, respectively. Slices that have these shapes with minimum positive weight are shown below. 

\ytableausetup{centertableaux}
\ytableaushort
{\none,\none }
* {4,3}
* [*(gray)]{4,2} ~ ($aq^1$), \ytableausetup{centertableaux}
\ytableaushort
{\none,\none }
* {5,2}
* [*(gray)]{4,2}
~ ($bq^1$),
\ytableausetup{centertableaux}
\ytableaushort
{\none,\none }
* {5,3}
* [*(gray)]{4,2}  
~ ($cq^2$),

\hfill

\ytableausetup{centertableaux}
\ytableaushort
{\none,\none }
* {4,4}
* [*(gray)]{4,2} ~ ($dq^2$),
\ytableausetup{centertableaux}
\ytableaushort
{\none,\none }
* {6,2}
* [*(gray)]{4,2} ~ ($eq^2$).

Slices with shapes $a$ and $b$ can only have odd weights, while slices with shapes $c, d$ and $e$ can only have even weights. Slice flow for the given profile is:
\begin{center}
\begin{tikzcd}
                                        & \color{blue}{dq^2} \arrow[rd]           &                                         & \color{blue}{dq^4} \arrow[rd]           &                                         & \color{blue}{dq^6} \arrow[rd]           &                                         & \color{blue}{dq^8} \\
\color{red}{aq^1} \arrow[ru] \arrow[rd] &                                         & \color{red}{aq^3} \arrow[ru] \arrow[rd] &                                         & \color{red}{aq^5} \arrow[ru] \arrow[rd] &                                         & \color{red}{aq^7} \arrow[ru] \arrow[rd] &                    \\
                                        & \color{red}{cq^2} \arrow[rd] \arrow[ru] &                                         & \color{red}{cq^4} \arrow[rd] \arrow[ru] &                                         & \color{red}{cq^6} \arrow[rd] \arrow[ru] &                                         & \color{red}{cq^8}  \\
\color{red}{bq^1} \arrow[ru] \arrow[rd] &                                         & \color{red}{bq^3} \arrow[ru] \arrow[rd] &                                         & \color{red}{bq^5} \arrow[ru] \arrow[rd] &                                         & \color{red}{bq^7} \arrow[ru] \arrow[rd] &                    \\
                                        & \color{blue}{eq^2} \arrow[ru]           &                                         & \color{blue}{eq^4} \arrow[ru]           &                                         & \color{blue}{eq^6} \arrow[ru]           &                                         & \color{blue}{eq^8}

\end{tikzcd}\\
 {\small Table 5.2: Profile $c=(2,2)$}
 \captionlistentry[table]{Profile $c=(2,2)$}
\end{center}
Before continuing, let us share with you a lemma which we will need for computing the infinite sums in this proof. Note that, the new lemma is another version of Lemma \ref{lemma8} with a small shift.
\begin{lemma} \label{lemma9}
For any integer $m_i \geq 1$ where $i=1,2, \ldots, n$, the following equality holds:
\begin{align*} 
&\sum_{k_n \geq 1} \ldots \sum_{k_2 \geq 1} \sum_{k_1 \geq 1} \bigg(\frac{q^{2k_1} q^{2k_1+2} \ldots q^{2k_1+2m_1-2}}{(1+q^{2k_1-1})(1+q^{2k_1+1})\ldots (1+q^{2k_1+2m_1-1})}\bigg) \bigg( \frac{q^{2k_1+2k_2+(2m_1)}}{(1+q^{2k_1+2k_2+2m_1-1})}\\
&\frac{q^{2k_1+2k_2+(2m_1+2)} \ldots q^{2k_1+2k_2+(2m_1+2m_2-2)}}{(1+q^{2k_1+2k_2+2m_1+1})\ldots (1+q^{2k_1+2k_2+2m_1+2m_2-1})} \bigg) \bigg( \frac{q^{2k_1+2k_2+2k_3+(2m_1+2m_2)}}{(1+q^{2k_1+2k_2+2k_3+2m_1+2m_2-1})}\\
&\frac{q^{2k_1+2k_2+2k_3+(2m_1+2m_2+2)} \ldots q^{2k_1+2k_2+2k_3+(2m_1+2m_2+2m_3-2)}}{(1+q^{2k_1+2k_2+2k_3+2m_1+2m_2+1})\ldots (1+q^{2k_1+2k_2+2k_3+2m_1+2m_2+2m_3-1})} \bigg) \ldots \\
&\bigg(\frac{ q^{2k_1+2k_2+\ldots+2k_n+(2m_1+2m_2+\ldots+2m_{n-1})} q^{2k_1+2k_2+\ldots+2k_n+(2m_1+2m_2+\ldots+2m_{n-1}+2)}}{ (1+q^{2k_1+2k_2+\ldots+2k_n+(2m_1+2m_2+\ldots+2m_{n-1}-1)})(1+q^{2k_1+2k_2+\ldots+2k_n+(2m_1+2m_2+\ldots+2m_{n-1}+1)})}\\
&\frac{\ldots q^{2k_1+2k_2+\ldots+2k_n+(2m_1+2m_2+\ldots+2m_{n}-2)}}{\ldots (1+q^{2k_1+2k_2+\ldots+2k_n+(2m_1+2m_2+\ldots+2m_{n-1}+2m_n-1)})}\bigg)\\
&\textbf{=}\frac{q^0 q^2 \ldots q^{2m_1+2m_2+ \ldots +2m_n-2}}{(1+q^1)(1+q^3)\ldots (1+q^{2m_1+2m_2+ \ldots +2m_n-1})} \times \\
&\frac{q^{2(m_1+m_2+\ldots+m_n)}}{(1-q^{2(m_1+m_2+\ldots+m_n)})} 
\frac{q^{2(m_2+m_3+\ldots+m_n)}}{(1-q^{2(m_2+m_3+\ldots+m_n)})}
\cdots \frac{q^{2(m_{n-1}+m_n)}}{(1-q^{2(m_{n-1}+m_n)})} \frac{q^{2m_n}}{(1-q^{2m_n})}. 
\end{align*}
\end{lemma}
Consider the following product in which the missing slices and the restrictions are indicated:
$$(1+aq^1+bq^1) \underbrace{(1+cq^2)}_{\underbrace{aq^1 \leftarrow dq^2 \rightarrow aq^3}_{bq^1 \leftarrow eq^2 \rightarrow bq^3}} (1+aq^3+bq^3) \underbrace{(1+cq^4)}_{\underbrace{aq^3 \leftarrow dq^4 \rightarrow aq^5}_{bq^3 \leftarrow eq^4 \rightarrow bq^5}} (1+aq^5+bq^5) \ldots .$$
Note that, this structure is the same as in Theorem \ref{theorem7} except for the first factor $(1+aq^1+bq^1)$ and for the missing slices and their restrictions corresponding to the second factor $(1+cq^2)$.   

If we allow slices to repeat and put $a=b=c=1$, the generating function of cylindric partitions including shapes $a$, $b$ or $c$ (all shapes except $d$ and $e$) is the following:
\begin{equation} \label{S(q) c=(2,2)}
S(q):= \frac{\prod_{k\geq0} (1+q^{2k+1})}{\prod_{k\geq0} (1-q^{k+1})}=\frac{(-q;q^2)_{\infty}}{(q;q)_{\infty}}.
\end{equation}
Let $n$ denote the number of distinct slices with shape $d$ or $e$ that appear in the cylindric partitions with the given profile. We put $d=e=1$. 

When $n=0$, the generating function of cylindric partitions with the given profile is \eqref{S(q) c=(2,2)}.

When $n=1$, we will either have a slice with shape $d$ or a slice with shape $e$. The minimum weights of slices with shapes $d$ and $e$ are the same which is $2$. So, we have the following generating function for case $n=1$:  
$$ \bigg(2 \cdot \sum_{k\geq 1} \frac{q^{2k}}{(1+q^{2k-1})(1+q^{2k+1})}\bigg)\cdot S(q).$$
This is nothing but
$$ \frac{2 \cdot q^2}{(1+q^1)(1-q^2)} \cdot S(q)=\frac{2\cdot q^{1(1+1)} \cdot (-q^2 ; q^2)_{1-1} }{(q^2 ; q^2)_1 \cdot (-q ; q^2)_1} \cdot \frac{(-q ; q^2)_{\infty}}{(q ; q)_{\infty}}.$$
When $n=2$, since the minimum distance between the weights of slices with shapes $d$ and $e$ is $4$, and slices with these shapes have the same minimum weight, for cases $(d,d)$ and $(e,e)$, also for cases $(d,e)$ and $(e,d)$ we have the same generating functions.

Case $(d,d)$: If weights of slices $(d,d)$ are $2k$ and $2k+2$ for $k\geq1$, the generating function including this duo is
$$\bigg( \sum_{k\geq 1} \frac{q^{2k} \cdot q^{2k+2}}{(1+q^{2k-1})(1+q^{2k+1})(1+q^{2k+3})}\bigg)\cdot S(q)=\frac{q^6}{(1+q^1)(1+q^3)(1-q^4)}\cdot S(q),$$
if weights of slices $(d,d)$ are $2k$ and $\geq2k+4$ for $k\geq1$, the generating function including this duo is

\begin{align*}
&\bigg(\sum_{m\geq1} \sum_{k\geq1} \frac{q^{2k} \cdot q^{2m+2k+2}}{(1+q^{2k-1})(1+q^{2k+1})(1+q^{2m+2k+1})(1+q^{2m+2k+3})}\bigg)\cdot S(q) \\
&=\frac{q^6}{(1+q^1)(1+q^3)(1-q^4)} \cdot \frac{q^2}{(1-q^2)}\cdot S(q).
\end{align*}
So, the generating function for case $(d,d)$ is the following:
$$\frac{q^6}{(1+q^1)(1+q^3)(1-q^2)(1-q^4)} \cdot S(q).$$
For case $(d, e)$, we have the following generating function:
\begin{align*}
&\bigg( \sum_{m\geq1} \sum_{k\geq 1} \frac{q^{2k} \cdot q^{2m+2k+2}}{(1+q^{2k-1})(1+q^{2k+1})(1+q^{2m+2k+1})(1+q^{2m+2k+3})}\bigg)\cdot S(q).\\
&=\frac{q^8}{(1+q^1)(1+q^3)(1-q^2)(1-q^4)}\cdot S(q).
\end{align*}
Therefore, the generating function for case $n=2$ is
$$\frac{2 \cdot q^6 (1 + q^2)}{(1+q^1)(1+q^3)(1-q^2)(1-q^4)} \cdot S(q)=\frac{2\cdot q^{2(2+1)} \cdot (-q^2 ; q^2)_{2-1} }{(q^2 ; q^2)_2 \cdot (-q ; q^2)_2} \cdot \frac{(-q ; q^2)_{\infty}}{(q ; q)_{\infty}}.$$
As before, we proceed by induction on $n$. The base case $n=1$ and case $n=2$ were done above.   

This time, when $n$ changes to $(n+1)$, the following happens: The number of indices ($k_i$'s) that we see changes from $n$ to $(n+1)$. So, by Lemma \ref{lemma9}, we multiply the denominator by $(1+q^{2n+1})(1-q^{2(n+1)})$. Recall that, in the numerator, we see the minimal weights coming from disjoint cases as a power of $q$. 

For a general $n$, the minimum weight comes from case $\underbrace{(d, d, d, \ldots, d)}_{\text{$n$ many}}$ is 
$$ 2 + 4 + 6 + \ldots + 2 \cdot n= n \cdot (n+1). $$
Starting with this case, how we compute the minimal weights for other cases is as follows: Let us work on $n=3$, as before. Starting with $(d, d, d)$, the minimal triple is $(dq^2, dq^4, dq^6)$. If we change the biggest slice with shape $d$ with a slice with shape $e$ (as the minimum distance between weights of slices with shapes $d$ and $e$ is $4$), it becomes $(dq^2, dq^4, eq^8)$, and we get  
$$\underbrace{(dq^2, dq^4,dq^6)}_{(q^2 \cdot q^4 \cdot q^6)} \rightarrow \underbrace{(dq^2, dq^4, dq^6), (dq^2, dq^4, eq^8)}_{(q^2 \cdot q^4 \cdot q^6) (1 + q^2)}.$$
Continuing with the cases $(dq^2, dq^4, dq^6), (dq^2, dq^4, eq^8)$, we will change the roles of $d$ and $e$ for the two biggest slices and we will get two more cases, that is, we will get $(dq^2, eq^6, eq^8), (dq^2, eq^6, dq^{10})$:
$$ \underbrace{(dq^2, dq^4, dq^6), (dq^2, dq^4, eq^8)}_{(q^2 \cdot q^4 \cdot q^6) (1 + q^2)} \rightarrow \underbrace{(dq^2, dq^4, dq^6), (dq^2, dq^4, eq^8),(dq^2, eq^6, eq^8), (dq^2, eq^6, dq^{10})}_{(q^2 \cdot q^4 \cdot q^6) (1 + q^2)(1+q^4)}.$$
All of the following cases are the cases that starts with a slice with shape $d$:  $(dq^2, dq^4, dq^6), (dq^2, dq^4, eq^8),$ $(dq^2, eq^6, eq^8), (dq^2, eq^6, dq^{10})$.

Changing the roles of $d$ and $e$ for above slices, we get the left cases: $(eq^2, eq^4, eq^6), (eq^2, eq^4, dq^{8}),$ $(eq^2, dq^6, dq^{8}), (eq^2, dq^6, eq^{10})$, then we obtain all the cases:

$(dq^2, dq^4, dq^6), (dq^2, dq^4, eq^8),(dq^2, eq^6, eq^8), (dq^2, eq^6, dq^{10}), \mathbf{(eq^2, eq^4, eq^6), (eq^2, eq^4, dq^{8})},$ \\
 $ \mathbf{(eq^2, dq^6, dq^{8}), (eq^2, dq^6, eq^{10})}$.

In this last step, since the roles of $d$ and $e$ are symmetric, no iteration is needed. We just multiply the previous step by $2$. So, finally we get:  
$$2 \cdot (q^2 \cdot q^4 \cdot q^6) (1 + q^2)(1+q^4).$$
We apply the same process for all $n$, to get the minimal weights. So, when $n$ changes to $(n+1)$, we multiply the numerator by $q^{2(n+1)} \cdot (1 + q^{2n})$.
Therefore, we obtain:
$$\bigg( \frac{2 \cdot q^{n\cdot(n+1)} \cdot (-q^2 ; q^2)_{n-1} \cdot q^{2(n+1)} \cdot (1 + q^{2n}) }{(-q ; q^2)_n \cdot (q^2 ; q^2)_n \cdot (1+q^{2n+1})(1-q^{2(n+1)})} \bigg) \cdot \frac{(-q ; q^2)_{\infty}}{(q ; q)_{\infty}}. $$
And, this is what we want:
$$\bigg(\frac{2 \cdot q^{(n+1)(n+2)} \cdot (-q^2 ; q^2)_{n} }{(-q ; q^2)_{n+1} \cdot (q^2 ; q^2)_{n+1}}\bigg) \cdot \frac{(-q ; q^2)_{\infty}}{(q ; q)_{\infty}}.$$
\end{proof}

\begin{theorem} \label{theorem9}
Cylindric partitions with profile $c=(3,1)$ have the following generating function:
$$\bigg(\sum_{n\geq 0} \frac{q^{n^2}}{(q^2; q^2)_n }\bigg) \cdot \frac{(-q^2 ; q^2)_{\infty}}{(q ; q)_{\infty}}. $$
\end{theorem}

\begin{proof}
Slices of cylindric partitions with profile $c = (3, 1)$ can have exactly $ \binom{\ell+r-1}{r-1}= $ $\binom{5}{1}$ shapes. The possible shapes are $(0), (1), (2), (3), (4)$ and we will denote these shapes by $a, c, b, d$ and $e$, respectively. Slices that have these shapes with minimum positive weight are shown below. 

\ytableausetup{centertableaux}
\ytableaushort
{\none,\none }
* {4,4}
* [*(gray)]{4,3} ~ ($aq^1$), \ytableausetup{centertableaux}
\ytableaushort
{\none,\none }
* {5,3}
* [*(gray)]{4,3}
~ ($bq^1$),
\ytableausetup{centertableaux}
\ytableaushort
{\none,\none }
* {5,4}
* [*(gray)]{4,3}  
~ ($cq^2$),

\hfill

\ytableausetup{centertableaux}
\ytableaushort
{\none,\none }
* {6,3}
* [*(gray)]{4,3} ~ ($dq^2$),
\ytableausetup{centertableaux}
\ytableaushort
{\none,\none }
* {7,3}
* [*(gray)]{4,3} ~ ($eq^3$).

Slices with shapes $a, b$ and $e$ can only have odd weights, while slices with shapes $c$ and $d$ can only have even weights. Slice flow for the given profile is the following:
\begin{center}
\begin{tikzcd}
\color{blue}{aq^1} \arrow[rd]           &                                         & \color{blue}{aq^3} \arrow[rd]           &                                         & \color{blue}{aq^5} \arrow[rd]           &                                         & \color{blue}{aq^7} \arrow[rd]           &                   \\
                                        & \color{red}{cq^2} \arrow[rd] \arrow[ru] &                                         & \color{red}{cq^4} \arrow[rd] \arrow[ru] &                                         & \color{red}{cq^6} \arrow[rd] \arrow[ru] &                                         & \color{red}{cq^8} \\
\color{red}{bq^1} \arrow[ru] \arrow[rd] &                                         & \color{red}{bq^3} \arrow[ru] \arrow[rd] &                                         & \color{red}{bq^5} \arrow[ru] \arrow[rd] &                                         & \color{red}{bq^7} \arrow[ru] \arrow[rd] &                   \\
                                        & \color{red}{dq^2} \arrow[ru] \arrow[rd] &                                         & \color{red}{dq^4} \arrow[ru] \arrow[rd] &                                         & \color{red}{dq^6} \arrow[ru] \arrow[rd] &                                         & \color{red}{dq^8} \\
                                        &                                         & \color{blue}{eq^3} \arrow[ru]           &                                         & \color{blue}{eq^5} \arrow[ru]           &                                         & \color{blue}{eq^7} \arrow[ru]           &               \end{tikzcd}\\
 {\small Table 5.3: Profile $c=(3,1)$}
 \captionlistentry[table]{Profile $c=(3,1)$}
\end{center}

Now, we will work on the following product in which the missing slices and the restrictions are indicated:
\begin{align*}
&\underbrace{(1+bq^1)}_{aq^1 \rightarrow (1 + cq^2)}\cdot(1+cq^2 +dq^2) \cdot \underbrace{(1+bq^3)}_{\underbrace{(1+cq^2) \leftarrow aq^3 \rightarrow (1+cq^4)}_{(1+dq^2) \leftarrow eq^3 \rightarrow (1+dq^4)}} \cdot (1+cq^4 +dq^4) \cdot\\
&\underbrace{(1+bq^5)}_{\underbrace{(1+cq^4) \leftarrow aq^5 \rightarrow (1+cq^6)}_{(1+dq^4) \leftarrow eq^5 \rightarrow (1+dq^6)}} \cdot (1+cq^6+dq^6) \underbrace{(1+bq^7)}_{\underbrace{(1+cq^6) \leftarrow aq^7 \rightarrow (1+cq^8)}_{(1+dq^6) \leftarrow eq^7 \rightarrow (1+dq^8)}} \ldots
\end{align*}
Now, let us define the following product which is the generating function of cylindric partitions created by slices with shapes $b, c$ or $d$, that is, shapes $a$ and $e$ are not included, when repetition of slices are not allowed:
\begin{equation} \label{P(q) c=(3,1)}
P(q):= \prod_{k \geq 0} (1 +bq^{2k+1})(1+cq^{2k+2}+dq^{2k+2}).
\end{equation}
In the product above, for the weight $1$ there is only one missing slice, namely $aq^1$ which can be followed by with slice $cq^2$. Except this case, for every weight $(2k+3)$, there are two missing slices: $aq^{2k+3}$ and $eq^{2k+3}$ for $k \geq 0$. If we put $aq^{2k+3}$ in place of $(1+bq^{2k+3})$ then we need to erase the slices $dq^{2k+2}$ and $dq^{2k+4}$, while if we put $eq^{2k+3}$ in place of $(1+bq^{2k+3})$ then we need to erase the slices $cq^{2k+2}$ and $cq^{2k+4}$. We will proceed as in the previous proofs of this section. Allowing slices to repeat and then putting $b=c=d=1$, \eqref{P(q) c=(3,1)} becomes
\begin{equation}  \label{S(q) c=(3,1)}
S(q):= \frac{\prod_{k\geq0} (1+q^{2k+2})}{\prod_{k\geq0}(1-q^{k+1})}=\frac{(-q^2;q^2)_{\infty}}{(q;q)_{\infty}}.
\end{equation}
Let $n$ denote the number of distinct slices with shapes $a$ or $e$ that appear (with any repetition) in the cylindric partitions with the given profile. As usual, we will proceed by induction on $n$. 

When $n=0$, the generating function of cylindric partitions with the given profile is \eqref{S(q) c=(3,1)}. 

When $n=1$, we will either have a slice with shape $a$ ($aq^{2k+1}$ for some $k\geq 0$), or a slice with shape $e$ ($eq^{2k+3}$ for some $k\geq 0$). So, the generating function for $n=1$ is
$$ \underbrace{\bigg(\frac{q^1}{(1+q^2)} + \sum_{k\geq1} \frac{q^{2k+1}}{(1+q^{2k})(1+q^{2k+2})}\bigg)\cdot S(q)}_{\text{every possible size of a slice with shape a included}} + \underbrace{\bigg(\sum_{k\geq1} \frac{q^{2k+1}}{(1+q^{2k})(1+q^{2k+2})}\bigg)\cdot S(q)}_{\text{every possible size of a slice with shape e included}}.$$
Using Lemma \ref{lem4}, for case $n=1$ we obtain:
$$\frac{q^1 (1+ q^2)}{(1+q^2)(1-q^2)} \cdot S(q)= \frac{q^{1^2}}{(q^2;q^2)_1} \cdot \frac{(-q^2;q^2)_{\infty}}{(q;q)_{\infty}}.$$
Suppose the following generating function holds for any $n$:
$$\bigg( \frac{q^{n^2}}{(q^2 ; q^2)_n}\bigg) \cdot \frac{(-q^2 ; q^2)_{\infty}}{(q ; q)_{\infty}}. $$
While $n$ changes to $(n+1)$, by Lemma \ref{lem4}, we multiply the denominator by $(1+q^{2(n+1)})(1-q^{2(n+1)})$ and in the numerator, we see the minimal weights coming from disjoint cases as a power of $q$. Since for a general $n$, the minimum weight comes from the case $\underbrace{(a, a, a, \ldots, a)}_{\text{$n$ many}}$ which is 
$$ 1 + 3 + 5 + \ldots + (2 n -1)= n^2, $$ 
while $n$ changes to $(n+1)$, we multiply the numerator by $q^{2n+1} \cdot (1 + q^{2(n+1)})$, and we obtain
$$\bigg( \frac{q^{n^2}}{(q^2 ; q^2)_n}\bigg) \cdot \frac{q^{2n+1}  (1 + q^{2(n+1)})}{(1+q^{2(n+1)})(1-q^{2.(n+1)})} \cdot \frac{(-q^2 ; q^2)_{\infty}}{(q ; q)_{\infty}}=\bigg( \frac{q^{(n+1)^2}}{(q^2 ; q^2)_{n+1}}\bigg) \cdot \frac{(-q^2 ; q^2)_{\infty}}{(q ; q)_{\infty}}.$$
\end{proof}

We will continue our work with the rank $4$ - level $2$ cylindric partitions. The proofs of the following theorems will be omitted, as the the proofs of Theorem \ref{theorem10}, Theorem \ref{theorem11}, and Theorem \ref{theorem12} are the same as the proofs of Theorem \ref{theorem7}, Theorem \ref{theorem8}, and Theorem \ref{theorem9}, respectively.

\begin{theorem} \label{theorem10}
Cylindric partitions with profile $c=(2,0,0,0)$, have the following generating function:

$$\bigg(\sum_{n\geq 0} \frac{q^{n(n+1)} \cdot (-q^2 ; q^2)_n }{(-q^3 ; q^2)_n \cdot (q^2 ; q^2)_n}\bigg) \cdot \frac{(-q^3 ; q^2)_{\infty}}{(q ; q)_{\infty}}. $$ 
\end{theorem}

\begin{proof}
Slices of cylindric partitions with profile $c = (2, 0, 0, 0)$ can have exactly $ \binom{\ell+r-1}{r-1}= $ $\binom{5}{3}$ shapes. The possible shapes are $(0,0,0)$, $(1,0,0)$, $(1,1,0)$, $(1,1,1)$, $(2,0,0)$, $(2,1,0)$, $(2,1,1)$, $(2,2,0)$, $(2,2,1)$, $(2,2,2)$. We will denote these shapes by $l, a, b, d, c, e, h, m, k$ and $n$, respectively. Slices that have these shapes with minimum positive weight are shown below. 

\ytableausetup{centertableaux}
\ytableaushort
{\none,\none }
* {3,2,2,2}
* [*(gray)]{2,2,2,2} ~ ($aq^1$), 
\ytableausetup{centertableaux}
\ytableaushort
{\none,\none }
* {3,3,2,2}
* [*(gray)]{2,2,2,2} ~ ($bq^2$), 
\ytableausetup{centertableaux}
\ytableaushort
{\none,\none }
* {4,2,2,2}
* [*(gray)]{2,2,2,2} ~ ($cq^2$), 
\ytableausetup{centertableaux}
\ytableaushort
{\none,\none }
* {3,3,3,2}
* [*(gray)]{2,2,2,2} ~ ($dq^3$), 

\hfill

\ytableausetup{centertableaux}
\ytableaushort
{\none,\none }
* {4,3,2,2}
* [*(gray)]{2,2,2,2} ~ ($eq^3$), 
\ytableausetup{centertableaux}
\ytableaushort
{\none,\none }
* {4,3,3,2}
* [*(gray)]{2,2,2,2} ~ ($hq^4$), 
\ytableausetup{centertableaux}
\ytableaushort
{\none,\none }
* {3,3,3,3}
* [*(gray)]{2,2,2,2} ~ ($lq^4$), 

\hfill

\ytableausetup{centertableaux}
\ytableaushort
{\none,\none }
* {4,4,2,2}
* [*(gray)]{2,2,2,2} ~ ($mq^4$), 
\ytableausetup{centertableaux}
\ytableaushort
{\none,\none }
* {4,4,3,2}
* [*(gray)]{2,2,2,2} ~ ($kq^5$), 
\ytableausetup{centertableaux}
\ytableaushort
{\none,\none }
* {4,4,4,2}
* [*(gray)]{2,2,2,2} ~ ($nq^6$). 

Slices with shapes $a$ and $k$ have weights congruent to $1$ $\pmod{4}$; slices with shapes $b, c$ and $n$ have weights congruent to $2$ $\pmod{4}$; slices with shapes $d$ and $e$ have weights congruent to $3$ $\pmod{4}$, finally slices with shapes $h, l$ and $m$ have weights congruent to $0$ $\pmod{4}$. Slice flow for the given profile is the same as the slice flow for profile $c=(4,0)$ and can be found below.
\begin{center}
\begin{tikzcd}
                                        & \color{blue}{cq^2} \arrow[rd]           &                                         & \color{blue}{mq^4} \arrow[rd]           &                                         & \color{blue}{nq^6} \arrow[rd]           &                                         & \color{blue}{lq^8} \\
\color{red}{aq^1} \arrow[ru] \arrow[rd] &                                         & \color{red}{eq^3} \arrow[ru] \arrow[rd] &                                         & \color{red}{kq^5} \arrow[ru] \arrow[rd] &                                         & \color{red}{dq^7} \arrow[ru] \arrow[rd] &                    \\
                                        & \color{red}{bq^2} \arrow[rd] \arrow[ru] &                                         & \color{red}{hq^4} \arrow[rd] \arrow[ru] &                                         & \color{red}{bq^6} \arrow[rd] \arrow[ru] &                                         & \color{red}{hq^8}  \\
                                        &                                         & \color{red}{dq^3} \arrow[ru] \arrow[rd] &                                         & \color{red}{aq^5} \arrow[ru] \arrow[rd] &                                         & \color{red}{eq^7} \arrow[ru] \arrow[rd] &                    \\
                                        &                                         &                                         & \color{blue}{lq^4} \arrow[ru]           &                                         & \color{blue}{cq^6} \arrow[ru]           &                                         & \color{blue}{mq^8}

\end{tikzcd}\\
 {\small Table 5.4: Profile $c=(2,0,0,0)$}
 \captionlistentry[table]{Profile $c=(2,0,0,0)$}
\end{center}

\end{proof}

\begin{theorem} \label{theorem11}
Cylindric partitions with profile $c=(1,0,1,0)$ have the following generating function:
$$\bigg(1 + \sum_{n\geq 1} \frac{2\cdot q^{n(n+1)} \cdot (-q^2 ; q^2)_{n-1} }{(q^2 ; q^2)_n \cdot (-q ; q^2)_n}\bigg) \cdot \frac{(-q ; q^2)_{\infty}}{(q ; q)_{\infty}}. $$
\end{theorem}

\begin{proof}
Slices of cylindric partitions with profile $c = (1, 0, 1, 0)$ can have exactly $ \binom{\ell+r-1}{r-1}= $ $\binom{5}{3}$ shapes. The possible shapes are $(0,0,0)$, $(1,0,0)$, $(1,1,0)$, $(1,1,1)$, $(2,0,0)$, $(2,1,0)$, $(2,1,1)$, $(2,2,0)$, $(2,2,1)$, $(2,2,2)$. We will denote these shapes by $e, h, l, b, n, a, d, c, f$ and $k$, respectively. Slices that have these shapes with minimum positive weight are shown below. 

\ytableausetup{centertableaux}
\ytableaushort
{\none,\none }
* {3,2,1,1}
* [*(gray)]{2,2,1,1} ~ ($aq^1$), 
\ytableausetup{centertableaux}
\ytableaushort
{\none,\none }
* {2,2,2,1}
* [*(gray)]{2,2,1,1} ~ ($bq^1$), 
\ytableausetup{centertableaux}
\ytableaushort
{\none,\none }
* {3,3,1,1}
* [*(gray)]{2,2,1,1} ~ ($cq^2$), 
\ytableausetup{centertableaux}
\ytableaushort
{\none,\none }
* {3,2,2,1}
* [*(gray)]{2,2,1,1} ~ ($dq^2$), 

\hfill

\ytableausetup{centertableaux}
\ytableaushort
{\none,\none }
* {2,2,2,2}
* [*(gray)]{2,2,1,1} ~ ($eq^2$), 
\ytableausetup{centertableaux}
\ytableaushort
{\none,\none }
* {3,3,2,1}
* [*(gray)]{2,2,1,1} ~ ($fq^3$), 
\ytableausetup{centertableaux}
\ytableaushort
{\none,\none }
* {3,2,2,2}
* [*(gray)]{2,2,1,1} ~ ($hq^3$), 

\hfill

\ytableausetup{centertableaux}
\ytableaushort
{\none,\none }
* {3,3,3,1}
* [*(gray)]{2,2,1,1} ~ ($kq^4$), 
\ytableausetup{centertableaux}
\ytableaushort
{\none,\none }
* {3,3,2,2}
* [*(gray)]{2,2,1,1} ~ ($lq^4$), 
\ytableausetup{centertableaux}
\ytableaushort
{\none,\none }
* {4,2,2,2}
* [*(gray)]{2,2,1,1} ~ ($nq^4$). 

Slices with shapes $a$ and $b$ have weights congruent to $1$ $\pmod{4}$; slices with shapes $c, d$ and $e$ have weights congruent to $2$ $\pmod{4}$; slices with shapes $f$ and $h$ have weights congruent to $3$ $\pmod{4}$, finally slices with shapes $k, l$ and $n$ have weights congruent to $0$ $\pmod{4}$. Slice flow for the given profile is the same as the slice flow for profile $c=(2,2)$ except for the number of shapes and their names, and can be found below.
\begin{center}
\begin{tikzcd}
                                        & \color{blue}{cq^2} \arrow[rd]           &                                         & \color{blue}{kq^4} \arrow[rd]           &                                         & \color{blue}{eq^6} \arrow[rd]           &                                         & \color{blue}{nq^8} \\
\color{red}{aq^1} \arrow[ru] \arrow[rd] &                                         & \color{red}{fq^3} \arrow[ru] \arrow[rd] &                                         & \color{red}{bq^5} \arrow[ru] \arrow[rd] &                                         & \color{red}{hq^7} \arrow[ru] \arrow[rd] &                    \\
                                        & \color{red}{dq^2} \arrow[rd] \arrow[ru] &                                         & \color{red}{lq^4} \arrow[rd] \arrow[ru] &                                         & \color{red}{dq^6} \arrow[rd] \arrow[ru] &                                         & \color{red}{lq^8}  \\
\color{red}{bq^1} \arrow[ru] \arrow[rd] &                                         & \color{red}{hq^3} \arrow[ru] \arrow[rd] &                                         & \color{red}{aq^5} \arrow[ru] \arrow[rd] &                                         & \color{red}{fq^7} \arrow[ru] \arrow[rd] &                    \\
                                        & \color{blue}{eq^2} \arrow[ru]           &                                         & \color{blue}{nq^4} \arrow[ru]           &                                         & \color{blue}{cq^6} \arrow[ru]           &                                         & \color{blue}{kq^8}

\end{tikzcd}\\
 {\small Table 5.5: Profile $c=(1,0,1,0)$}
 \captionlistentry[table]{Profile $c=(1,0,1,0)$}
\end{center}
\end{proof}

\begin{theorem} \label{theorem12}
Cylindric partitions with profile $c=(1,1,0,0)$ have the following generating function:
$$\bigg(\sum_{n\geq 0} \frac{q^{n^2}}{(q^2 ; q^2)_n }\bigg) \cdot \frac{(-q^2 ; q^2)_{\infty}}{(q ; q)_{\infty}}. $$
\end{theorem}

\begin{proof}
Slices of cylindric partitions with profile $c = (1, 1, 0, 0)$ can have exactly $ \binom{\ell+r-1}{r-1}= $ $\binom{5}{3}$ shapes. The possible shapes are $(0,0,0)$, $(1,0,0)$, $(1,1,0)$, $(1,1,1)$, $(2,0,0)$, $(2,1,0)$, $(2,1,1)$, $(2,2,0)$, $(2,2,1)$, $(2,2,2)$. We will denote these shapes by $f, k, a, d, b, c, e, h, l$ and $n$, respectively. Slices that have these shapes with minimum positive weight are shown below. 

\ytableausetup{centertableaux}
\ytableaushort
{\none,\none }
* {2,2,1,1}
* [*(gray)]{2,1,1,1} ~ ($aq^1$), 
\ytableausetup{centertableaux}
\ytableaushort
{\none,\none }
* {3,1,1,1}
* [*(gray)]{2,1,1,1} ~ ($bq^1$), 
\ytableausetup{centertableaux}
\ytableaushort
{\none,\none }
* {3,2,1,1}
* [*(gray)]{2,1,1,1} ~ ($cq^2$), 
\ytableausetup{centertableaux}
\ytableaushort
{\none,\none }
* {2,2,2,1}
* [*(gray)]{2,1,1,1} ~ ($dq^2$), 

\hfill

\ytableausetup{centertableaux}
\ytableaushort
{\none,\none }
* {3,2,2,1}
* [*(gray)]{2,1,1,1} ~ ($eq^3$),
\ytableausetup{centertableaux}
\ytableaushort
{\none,\none }
* {2,2,2,2}
* [*(gray)]{2,1,1,1} ~ ($fq^3$),
\ytableausetup{centertableaux}
\ytableaushort
{\none,\none }
* {3,3,1,1}
* [*(gray)]{2,1,1,1} ~ ($hq^3$),
\ytableausetup{centertableaux}
\ytableaushort
{\none,\none }
* {3,2,2,2}
* [*(gray)]{2,1,1,1} ~ ($kq^4$),

\hfill

\ytableausetup{centertableaux}
\ytableaushort
{\none,\none }
* {3,3,2,1}
* [*(gray)]{2,1,1,1} ~ ($lq^4$),
\ytableausetup{centertableaux}
\ytableaushort
{\none,\none }
* {3,3,3,1}
* [*(gray)]{2,1,1,1} ~ ($nq^5$).

Slices with shapes $a, b$ and $n$ have weights congruent to $1$ $\pmod{4}$; slices with shapes $c$ and $d$ have weights congruent to $2$ $\pmod{4}$; slices with shapes $e, f$ and $h$ have weights congruent to $3$ $\pmod{4}$, finally slices with shapes $k$ and $l$ have weights congruent to $0$ $\pmod{4}$. Slice flow for the given profile is the same as the slice flow for profile $c=(3,1)$ except for the number of shapes and their names, and can be found below.

\begin{center}
\begin{tikzcd}
\color{blue}{bq^1} \arrow[rd]           &                                         & \color{blue}{hq^3} \arrow[rd]           &                                         & \color{blue}{nq^5} \arrow[rd]           &                                         & \color{blue}{fq^7} \arrow[rd]           &                   \\
                                        & \color{red}{cq^2} \arrow[rd] \arrow[ru] &                                         & \color{red}{lq^4} \arrow[rd] \arrow[ru] &                                         & \color{red}{dq^6} \arrow[rd] \arrow[ru] &                                         & \color{red}{kq^8} \\
\color{red}{aq^1} \arrow[ru] \arrow[rd] &                                         & \color{red}{eq^3} \arrow[ru] \arrow[rd] &                                         & \color{red}{aq^5} \arrow[ru] \arrow[rd] &                                         & \color{red}{eq^7} \arrow[ru] \arrow[rd] &                   \\
                                        & \color{red}{dq^2} \arrow[ru] \arrow[rd] &                                         & \color{red}{kq^4} \arrow[ru] \arrow[rd] &                                         & \color{red}{cq^6} \arrow[ru] \arrow[rd] &                                         & \color{red}{lq^8} \\
                                        &                                         & \color{blue}{fq^3} \arrow[ru]           &                                         & \color{blue}{bq^5} \arrow[ru]           &                                         & \color{blue}{hq^7} \arrow[ru]           &               \end{tikzcd}\\
 {\small Table 5.6: Profile $c=(1,1,0,0)$}
 \captionlistentry[table]{Profile $c=(1,1,0,0)$}
\end{center}
\end{proof}

\section{Future Work}
\label{sec:future}
Our plan is to continue to work on cylindric partitions with rank $2$ and levels $\geq 5$. We will try to generalize our results. We will also simultaneously work on cylindric partitions with profile $c=(1,1,1)$. We will try to give expressions alternative to Borodin's formula for the generating functions having the profiles mentioned above. 
		
\section*{Acknowledgements}  
This work is based on the author's PhD dissertation. The author would like to thank her advisor Kağan Kurşungöz for his valuable guidance and comments. The author also would like to thank Ole Warnaar for reading the manuscript and suggesting the references \cite{KanadeRussell}, \cite{uncu2023proofs},  \cite{Warnaar2025} and telling the connection between formula \eqref{c=(3,1)}  and formula (7.25) in \cite{warnaar2023a2}.

	\bibliography{refs}
	
\end{document}